\documentclass[11pt,a4paper]{article}
\usepackage{float}
\usepackage{amsmath,amssymb}
\usepackage{hyperref}
\usepackage{amsthm}
\usepackage{algorithm}
\usepackage[noend]{algpseudocode}
\usepackage{authblk}
\usepackage{epsfig}
\usepackage{pst-all} 
\usepackage{pst-plot} 
\usepackage[space]{grffile} 
\usepackage{etoolbox} 
\makeatletter 
\patchcmd\Gread@eps{\@inputcheck#1 }{\@inputcheck"#1"\relax}{}{}
\makeatother

\hypersetup{
    bookmarks=true,
    unicode=false,
    pdftoolbar=true,
    pdfmenubar=true,
    pdffitwindow=false,
    pdfstartview={FitH},
    pdfnewwindow=true,
    colorlinks=true,
    linkcolor=blue,
    citecolor=blue,
    filecolor=magenta,
    urlcolor=blue
}

\newcommand{\be}{\begin{equation}}
\newcommand{\ee}{\end{equation}}

\theoremstyle{plain}

\newtheorem{definition}{Definition}[section]
\newtheorem{theorem}[definition]{Theorem}
\newtheorem{lemma}[definition]{Lemma}
\newtheorem{corollary}[definition]{Corollary}
\newtheorem{prop}[definition]{Proposition}

\theoremstyle{definition} 
\newtheorem{remark}[definition]{Remark}

\begin{document}

\title{A row-sampling based randomised finite element method for elliptic partial differential equations} 
\author[$\dagger$ $\flat$]{Yue Wu\thanks{The corresponding author: yue.wu@ed.ac.uk, yue.wu@maths.ox.ac.uk}} 
\author[$\dagger$]{Dimitris Kamilis}
\author[$\dagger$ $\ddagger$]{Nick Polydorides}
\affil[$\dagger$]{School of Engineering, University of Edinburgh, Edinburgh, UK}
\affil[$\flat$]{Mathematical Institute, University of Oxford, Oxford, UK}
\affil[$\ddagger$]{The Alan Turing Institute, London, UK}
\maketitle

\section*{Abstract}

 We consider a randomised implementation of the finite element method (FEM) for elliptic partial differential equations on high-dimensional models. This is motivated by applications where model predictions are essential for real-time process diagnostics. In these circumstances it is imperative to expedite prediction without a significant compromise in the model's fidelity, which in turn relies on the rapid assembly and solution of the associated system of equations typically at the many-query context. Our approach involves converting the solution of the linear, symmetric positive definite FEM system into an over-determined least squares problem, whose solution is then projected onto a low-dimensional subspace. The resulting low-dimensional system can be effectively sketched as a product of two high-dimensional matrices using a parameter-dependent non-uniform sampling distribution, utilising only a small subset of the model's parameters. Although different to the optimal sampling distributions based on the statistical leverage-scores of the rows of the matrices, we show that the distance between them shrinks for an appropriate choice of the projection subspace. For the approximate solution we bound the incurring errors due to the projection, subspace approximation and sketching and show that the overall error is dominated by the condition number of the projected stiffness matrix. Our approach is tested on simulations on the Dirichlet and Neumann problems for the steady-state diffusion equation. The results show that our approach has on average a tenfold improvement on the computational times compared to the classical deterministic framework at the expense of a moderately small approximation error.

\section{Introduction}

We consider the implementation of the Finite Element Method (FEM) in high-dimensional discrete models associated with elliptic partial differential equations, focusing in particular to the many-query context, where an approximate solution is sought for various inhomogeneous parameter fields. Owing to its versatility in handling models of realistic complexity, the method has been at the forefront of numerical computing and simulation for electromagnetic, mechanical, heat transfer and fluid dynamics systems \cite{ElmanSilvesterWathen}. Beyond its appeal in applied engineering research, the method has led to several algorithmic advances in scientific computing such as matrix preconditioning, fast iterative algorithms and multigrid methods \cite{Saad}. 

Our work is motivated by the need to expedite model prediction, also referred to as forward problem evaluation, in the context of a FEM-based simulation in the cases where an approximate, yet fast solution is imperative. Realising fast, real-time simulation, with large three-dimensional models is a formidable task and yet it can be of critical importance in a number of instances like online calibration of a sensor network or the control of a manufacturing process, where accurate, expensive simulations are typically deferred offline on specialised high performance computing infrastructure. Reducing the computing time for forward evaluations has been a long-standing goal for model-order reduction in computational partial differential equations and the main bottleneck of statistical inference algorithms for inverse problems and Bayesian uncertainty quantification, where multiple model runs are sought in the many query or Monte Carlo simulation context \cite{BennerCohenWillcox}, \cite{LordPowellShardlow}. It is worth emphasising however, that in practical applications involving experimental data contaminated with noise, an approximate evaluation to the respective forward problem suffices for the purpose of making model-based inferences with such data \cite{Beskos}. In this context, an approximate solution to an accurate model is preferable to an
accurate solution to an oversimplified model as the former typically allows to quantify and control the model-induced error that's otherwise hard to estimate \cite{Calvetti}.


When the accuracy of the solution can be traded off against speed, algorithms based on randomised linear algebra present a competitive alternative \cite{Woodruff}. The connection between this framework and numerical computing goes back to the sketching approach of Drineas and Mahoney for the Laplacian of a graph, where they coined the relationship between statistical leverage scoring and the so-called effective resistance of the graph \cite{DrineasMahoneyER}. Although the method is suitable to symmetric diagonally dominant (SDD) linear systems while the FEM systems are typically not SDD, there is an evocative similarity in the structure of the coefficient matrices in the respective systems, particularly in solving elliptic partial differential equations, where FEM leads to the so-called stiffness matrix, a generalisation of the Laplacian paradigm discretised on unstructured grids. The concept of effective resistances has since led to sketch-based preconditioners for SDD systems through sparsifier algorithms aimed at reducing the matrix fill-in and thus render the resulting systems solvable in a time that is asymptotically linear to their sparsity level \cite{KoutisMillerPeng}, \cite{SpielmanTeng}. More recently, Avron and Toledo have proposed a generalisation of this framework to the FEM context adapting the idea of effective resistance to that of the effective stiffness of an element in the grid \cite{AvronToledo}, relaxing the restriction to SDD systems. In particular, for the FEM sparse symmetric positive definite (SSPD) matrices, they derive formulas for the effective stiffness and show their equivalence to the statistical leverage scores, claiming that sampling $O(n\log n)$ elements according to those can lead to a sparser preconditioner such that the resulting system is solvable, with high-probability, in a small number of iterations.

While the above approaches focus predominantly on the efficient preconditioning and assembling of such systems, randomised algorithms for large-scale linear systems have already been proposed and implemented. The framework of Gower and Richtarik for example randomises the row-action iterative methods by taking random projections onto convex sets \cite{GowerRichtarik}. Applied to the FEM-induced SSPD systems, the underpinning algorithm is equivalent to a stochastic gradient descent method with provable convergence, while the approach in \cite{GowerRichtarik2} iteratively sketches the inverse of a matrix. Besides, there is a wealth of literature on sketching methods for least-squares problems, constrained or unconstrained, using data-oblivious subspace embeddings (randomised sketching transforms) that preserve some approximate isometry and orthogonality in the sketched systems. We refer the reader to the work of Woodruff \cite{Woodruff}, Drineas and Mahoney \cite{DMMS}, Pilanci and Wainwright \cite{Pilanci}, and Boutsidis and Drineas. \cite{BoutsidisDrineas}.    
  
In \cite{BertsekasYu09}, Bertsekas and Yu present an alternative approach for simulating an approximate solution to linear fixed-point equations and least squares problems, in the context of evaluating the cost of stationary policies in a Markovian decision. This is based on approximate dynamic programming algorithms that solve a projected form of Bellman's equation in a low-dimensional subspace, using sample-based approximations. Subsequently this framework was extended and coupled with importance sampling schemes by Polydorides et al. \cite{PolydoridesWangBertsekas} in solving linear inverse problems associated with Fredholm integral equations of the first kind, exploiting the characteristic smooth structure of the integral kernels.

In the many-query context one faces two computational challenges: the fast assembly of the large FEM system for each query (parameter vector), and the efficient solution of the resulting FEM system to some level of accuracy. We begin by transforming the linear SSPD FEM system into an over-determined least squares problem, and then apply a deterministically chosen orthogonal projection onto a low-dimensional subspace. Our efforts then focus on the efficient randomisation of the projected least-squares equations for every parameter query, by extending ideas from \cite{DrineasMahoneyER} and \cite{BertsekasYu09}. In this context, our contributions are in the development of the projected randomisation algorithm, the analysis of the impact of the projection on the approximation of the leverage scores, and the derivation of error bounds for the sketched projected solution. Further, we implement the proposed algorithm on Dirichlet and Neumann problems for the elliptic diffusion partial differential equation.

Our paper is organised as follows: In the next section we provide a brief introduction to the Galerkin FEM formulation for elliptic boundary value problems. We then derive the subspace-projected formulation and then we describe the sketching algorithm. Subsequently, we investigate the distance between the adopted and optimal sampling distributions in the context of the subspace projection, and then we conclude with an analysis of the various types of errors imparted on the solution through the various stages of the methodology. We end our report with a presentation and discussion of some numerical results. Whenever suitable we delegate the proofs to the appendix.   

\section{Finite element method preliminaries}

We consider the elliptic partial differential equation 
\be\label{pde}
 - \nabla \cdot p(x) \nabla u(x) = f(x) \quad \text{in}\; \Omega \subset \mathbb{R}^3,
\ee
and associated boundary conditions 
\be\label{bnd}
u = g^{(D)} \; \text{on}\; \partial\Omega_D \quad \text{and} \quad \nabla u \cdot \hat{n} = g^{(N)} \; \text{on}\; \partial\Omega_N,
\ee
on a bounded and simply connected domain $\Omega \subset \mathbb{R}^d$, $d=2,3$ with a Lipschitz smooth boundary $\partial \Omega = \partial \Omega_D \cup \partial \Omega_N$, and $\hat{n}$ the unit normal on the boundary. Further let $p(x)$ be a real, scalar and positive parameter function supported over the closure of the domain 
\be\label{padmi}
0 < p_{\mathrm{min}} \leq p(x) \leq p_{\mathrm{max}} < \infty, \quad x \in \Omega,
\ee
where $x\doteq(x_1,\ldots,x_{d})$ denotes the spatial coordinate vector. In this work we consider primarily the three-dimensional case $(d=3)$ but whenever possible we keep the notation general. Multiplying \eqref{pde} by an appropriate test function $v$, then integrating over the domain and invoking the divergence theorem yields
\be\label{variational}
\int_\Omega \mathrm{d}x \, \nabla u \cdot p \nabla v = \int_\Omega \mathrm{d}x \, f v + \int_{\partial \Omega} \mathrm{d}s \,g^{(N)} v,
\ee 
where $\mathrm{d}x$ and $\mathrm{d}s$ are volume and surface integration elements respectively. Using the standard definition of the Sobolev space on this domain as 
\be
\mathcal{H}^1(\Omega) \doteq \Bigl \{u(\Omega) \Bigl | u, \frac{\partial u}{\partial x_q}, q=1,\ldots,d \in L^2(\Omega) \Bigr \},
\ee
where $L^2(\Omega)$ is the space of square-integrable functions on $\Omega$ we can define the solution and test function spaces as
\be
\mathcal{H}^1_U \doteq \Bigl \{ u \in \mathcal{H}^1(\Omega) \Bigl | u = g^{(D)} \; \text{on}\; \partial \Omega_D \Bigr \},
\quad
\mathcal{H}^1_0 \doteq \Bigl \{ v \in \mathcal{H}^1(\Omega) \Bigl | v = 0 \; \text{on}\; \partial \Omega_D \Bigr \}
\ee
respectively. Assuming $p \in L^\infty(\bar \Omega)$ and $f \in L^2(\Omega)$ are in the Banach spaces of real functions defined on the closure of the domain $\bar \Omega$ and its interior respectively, and similarly $g^{(D)} \in H^{\frac 12}(\partial \Omega_D)$, $g^{(N)} \in H^{-\frac 1 2}(\partial \Omega_N)$, the weak form of the boundary value problem \eqref{pde}-\eqref{bnd} is to find a function $u \in \mathcal{H}^1_U$ such that 
\be\label{weak}
\int_\Omega \mathrm{d}x \, \nabla u \cdot p \nabla v = \int_\Omega \mathrm{d}x \, f v + \int_{\partial \Omega} \mathrm{d}s \, g^{(N)} v, \qquad \forall v \in \mathcal{H}^1_0.
\ee
In these conditions the existence and uniqueness of the weak solution is guaranteed by the Lax-Milgram theorem \cite{ElmanSilvesterWathen}.

To derive the Galerkin finite element approximation method from the weak form \eqref{weak}, we consider $\mathcal{T}_\Omega \doteq \{\Omega_1, \ldots, \Omega_k\}$ a tetrahedral mesh and $\mathcal{S}^1_\Omega \subset \mathcal{H}^1_0$ the conforming finite dimensional space associated with the chosen finite element basis defined on $\mathcal{T}_\Omega$. Let us explicitly quote also $\Delta_\Omega \doteq \{{\partial \Omega}_1,\ldots, {\partial \Omega}_\tau\}$ the set of $\tau$ triangular faces (resp. straight edges in $d=2$) spanning the outer surface of the discrete domain so that $\bigcup_{\ell=1}^k \Omega_\ell \approx \Omega$ and $\bigcup_{\ell=1}^\tau {\partial \Omega}_\ell \approx \partial \Omega$. The notations $|\Omega|$ and $|\partial \Omega|$ are used to express the volume (resp. area) and boundary area (resp. length) of the domain respectively. In particular, we denote the subset of $\Delta_\Omega$ on the Neumann boundary as $\Delta_\Omega^N$. If $\mathcal{S}^1_\Omega \doteq \text{span}\{\phi_1(x),\ldots, \phi_{n}(x), \ldots, \phi_{n+n_\partial}(x) \}$ comprises of piecewise linear shape functions with local support over the elements in $\mathcal{T}_\Omega$ then we can express the FEM approximation of the potential as   
\be\label{uexp}
u_h = \sum_{i=1}^{n} u_i \phi_i + \sum_{i=n+1}^{n+n_\partial} \ u_i \phi_i,
\ee
separating the expansion between the functions defined on the $n$ interior and $n_\partial$ boundary nodes. From this, the finite element formulation of the boundary value problem is to find $u_h \in \mathcal{S}^1_\Omega$ such that
\be
\sum_{\Omega_\ell \in \mathcal{T}_\Omega} \int_{\Omega_\ell} \mathrm{d}x\, \nabla u_h \cdot p \nabla u_h = \sum_{\Omega_\ell \in \mathcal{T}_\Omega} \int_{\Omega_\ell} \mathrm{d}x\, f v_h + \sum_{{\Omega_\partial}_\ell \in \Delta_\Omega^N} \int_{{\Omega_\partial}_\ell} \mathrm{d}s \, g^{(N)} v_h, \forall v_h \in \mathcal{S}^1_\Omega,
\ee 
where $g^{(N)}$ is the Neumann function on $\partial \Omega_N$. Further we select a piecewise constant basis of characteristic functions $\{\chi_1, \ldots, \chi_k\}$, where $\chi_\ell=1$ over $\Omega_\ell$ and zero elsewhere, so that the parameter and forcing terms\footnote{This choice of basis is not restrictive although it simplifies the notation and the calculations. Alternatively, one could take for example $$f_\ell = \frac{1}{|\Omega_\ell|}\int_{\Omega_\ell} \mathrm{d}x\, f$$ and compute the volume integrals encountered in the Galerkin formulation \eqref{gal} using numerical quadrature rules.} are expressed as
\be\label{pwc}
p_h = \sum_{\ell=1}^{k} p_\ell \chi_\ell, \quad \text{and} \quad f_h = \sum_{\ell=1}^{k} f_\ell \chi_\ell. 
\ee
We then write the Galerkin system of equations for the vector $\{u_1,\ldots, u_{n+n_\partial}\}$
\be\label{gal}
\sum_{j=1}^{n + n_\partial} u_j \sum_{\Omega_\ell \in \mathcal{T}_\Omega} \int_{\Omega_\ell} \mathrm{d}x\, \nabla \phi_i \cdot p_\ell \nabla \phi_j = \sum_{\Omega_\ell \in \mathcal{T}_\Omega} \int_{\Omega_\ell} \mathrm{d}x\, f_\ell \phi_i + \sum_{{\Omega_\partial}_\ell \in \Delta_\Omega^N} \int_{{\Omega_\partial}_\ell} \mathrm{d}s \, {g^{(N)}}_\ell \phi_i ,
\ee
for $i=1,\ldots, n+n_\partial$. Note that in the instance of the Dirichlet problem where $\partial \Omega_N = \emptyset$, the surface integral vanishes and the coefficients $\{u_{n+1},\ldots,u_{n+n_\partial}\}$ are fixed through $g^{(D)}$, hence the Galerkin system of equations has $n$ degrees of freedom, while for the Neumann problem $u$ has dimension $n+n_\partial - 1$, after applying the uniqueness condition. The assembly of \eqref{gal} over the elements in the domain yields a system
\be\label{aub}
A u = b,
\ee
where $A \in \mathbb{R}^{n+n_\partial \times n+n_\partial}$, the so-called FEM stiffness matrix, that is sparse, symmetric and positive-definite. The FEM construction guarantees that $b \in \mathbb{R}^{n+n_\partial}$ is in the column space of $A$ therefore the system \eqref{aub} admits a unique solution $u^* = A^{-1}b$. The focus of our work is the efficient approximation of $u^*$ in the many $p$ query context, such as the one used in Monte-Carlo approaches for inverse problems \cite{Beskos}. As such our approach will be faced with two main challenges: the efficient assembly of the stiffness matrix, and thereafter the speedy solution of the resulted FEM problem. For completeness, we define our target problem as follows.  
\begin{definition}\label{probdef}
If $p^{(1)}(x),\ldots, p^{(N)}(x)$ are parameter functions corresponding to the boundary value problem \eqref{pde}-\eqref{bnd} with fixed boundary and forcing conditions and  $A^{(1)}, \ldots, A^{(N)}$ the respective FEM stiffness matrices, compute the approximate solutions $u^{(i)}$ of  
$$
A^{(i)} u^{(i)} = b, \quad for \quad i=1,\ldots, N,
$$
where $N$ and the dimensions of $A$ are large.
\end{definition}

\subsection{Notation}

Hereafter, in a discrete model $\mathcal{T}_\Omega$ with $k$ elements and $n + n_\partial$ nodes we express as $p_\ell$ the $\ell$th element of the positive parameter vector $p \in \mathbb{R}^k$, $|\Omega_\ell|$ the volume or area of that element, $\omega$ the vector $\{|\Omega_1|,\ldots, |\Omega_k|\}$ and $u\in \mathbb{R}^{n + n_\partial}$ the sought FEM solution coefficients.
For a matrix $X$, $X_{(\ell)*}$ and $X_{*(\ell)}$ denote the $\ell$th row and column of $X$ respectively, $X_{i j}$ its $i,j$th element, $X_{i,(j:j+n)}$ the elements on the $i$th row between columns $j$ and $j+n$ and $X_{*(i:j)}$ the part of the matrix in columns $i$ to $j$. $X^\dag$ is the pseudo-inverse of $X$, $\kappa(X)$ its condition number, $\sigma_i(X)$ its $i$th singular value, and $\lambda_i(X)$ its corresponding eigenvalue. For $X \in \mathbb{R}^{m \times n}$ with $m \geq n$ we define the singular value decomposition $X = U_X \Sigma_X V_X^T$ where $U_X \in \mathbb{R}^{m \times n}$, $\Sigma_X \in \mathbb{R}^{n \times n}$ and $V_X \in \mathbb{R}^{n \times n}$. Unless stated otherwise, singular values and eigenvalues are ordered in non-increasing order, thus for a square matrix $X \in \mathbb{R}^{n \times n}$, $\lambda_1(X)=\lambda_{\max}(X)$ is the largest eigenvalue, $\lambda_n(X)=\lambda_{\min}(X)$ the smallest, and $\mathrm{diag}(X)$ is the vector of its main diagonal. Further we write $\|\cdot\|$ for the Euclidean norm for a vector or the spectral norm of a matrix, $\|\cdot\|_{\max}$ the max norm of a vector and $\| \cdot \|_F$ the Frobenius norm of a matrix. For matrices $X$ and $Y$ with the same number of rows $(X|Y)$ is the augmented matrix formed by column concatenation. The notation $I$ is for the identity matrix, or when dimension is important to the context $I_n$ is the identity in dimension $n$, and $[n]$ is the set of integers from 1 to $n$ inclusive. For two scalar quantities $a$ and $b$, $a \vee b$ denotes the maximum of $a$ and $b$. Finally, $\mathbb{E}_{\xi}[\cdot]$ stands for the expectation of a random scalar or matrix under probability density $\xi$, and $\mathrm{Var}_{\xi}[\cdot]$ for the variance of an estimator under $\xi$.  

\subsection{Assembly of the stiffness matrix}

From the definitions of the shape functions in \eqref{uexp}, forcing terms and Neumann boundary conditions, the element of the stiffness matrix is given by \be\label{stiff}
A_{ij} =  \sum_{\Omega_\ell \in \mathcal{T}_\Omega} |\Omega_\ell| \, p_\ell \, \nabla \phi_i \cdot \nabla \phi_j, \quad i,j \in \mathcal{I}_\ell,
\ee 
where $\mathcal{I}_\ell$ is the index set of the $d+1$ vertices of the $\ell$th element. Forming the sparse matrix $D_\ell \in \mathbb{R}^{d\times (n+n_\partial)}$ with the gradients of the shape functions defined at $\mathcal{I}_\ell$ and stacking them together for all $k$ elements to a matrix $D \in \mathbb{R}^{kd \times (n+n_\partial)}$ we can define 
\be\label{Ydef}
Y = Z^{\frac 1 2} D
\ee
for a positive diagonal matrix 
\be
Z = z \otimes I_d, 
\ee
where the vector $z$ has elements $z_\ell=|\Omega_\ell|p_\ell$ and $\otimes$ denotes the Kronecker product. Intrinsically, the FEM construction allows forming the stiffness matrix either as a high-dimensional sum 
\be\label{sumstiff}
A = \sum_{\ell=1}^k Y_\ell^T Y_\ell,  \quad \text{where} \quad Y_\ell =\sqrt{z_\ell} D_\ell,
\ee   
or a matrix product  
\be \label{Adec}
A = Y^TY =\sum_{\ell=1}^k \sum_{j=0}^{d-1} \sum_{j'=0}^{d-1} Y^T_{*(3\ell-j)}Y_{(3\ell-j')*},
\ee
both of which require an efficient assembly using reference elements and geometry mappings \cite{KirbyLogg}.
For simplicity in the notation we demonstrate our methodology by considering the Dirichlet problem where $A$ has dimensions $n \times n$, and note that modifications to the Neumann problem are trivial following the conventional FEM implementation \cite{ElmanSilvesterWathen}. In our approach we follow the product construction \eqref{Adec}, for which the spectrum of the stiffness matrix $A$, and respectively that of $Y$, are important. For completeness we quote two relevant bounds from \cite{KamenskiHuangXu}.
\begin{lemma}\label{largesteig}
For $Y = Z^{\frac 1 2}D$ with a singular value decomposition (SVD) $Y = U_Y \Sigma_Y V_Y^T$, then the largest eigenvalue of the stiffness matrix $A$ is $\lambda_1(\Sigma_Y)^2$ and it is bounded by
\be
\max_i A_{ii} \leq \lambda_1(\Sigma_Y)^2 \leq (d+1) \max_i A_{ii}
\ee
\end{lemma}
\begin{proof}
The proof is in Lemma 4.1 of \cite{KamenskiHuangXu}.
\end{proof}

\begin{lemma}\label{smallesteig}
For $Y = Z^{\frac 1 2}D$ with SVD $Y = U_Y \Sigma_Y V_Y^T$, then the smallest eigenvalue of the stiffness matrix $A$ is bounded from below by
\be
\lambda_n(\Sigma_Y)^2 \geq 
C p_{\mathrm{min}}\frac 1 k \begin{cases} \Bigl (1 + \log \frac{\bar \omega}{\omega_{\mathrm{min}}} \Bigr)^{-1}, & d=2\\
\Bigl ( \frac 1 k \sum_{\Omega_\ell \in \mathcal{T}_\Omega} \bigl ( \frac{\bar \omega}{\omega_\ell} \bigr )^{\frac 1 2} \Bigr )^{-\frac 2 3}, & d=3
\end{cases}
\ee
where $\bar \omega$ is the average element size, and $\omega_{\mathrm{min}}$ is the minimum element size in the mesh. $C$ is a generic constant, $k$ the total number of elements in the mesh and $p_{\textrm{min}}$ the minimum value in the parameter vector.
\end{lemma}
\begin{proof}
The proof is in Lemma 5.1 of \cite{KamenskiHuangXu}.
\end{proof}

\subsection{Dimensionality reduction}\label{sec:DimRed}

Let us recall from \eqref{Adec} and the definition $Y = Z^{\frac 12} D$ that the dependence of the stiffness matrix $A$ on the parameter vector $p$ is restricted to the diagonal $Z$. It can thus be shown that the solution of the consistent system of the FEM equations $Au=b$ can be alternatively obtained by solving the over-determined least squares problem
\be\label{uLS}
\hat u_{\text{ls}} = \arg\min_{u \in \mathbb{R}^n} \|Y u - Z^{-\frac 1 2} (D^T)^\dag b\|^2,
\ee
Assuming that the inverse of $A$ exists, this is immediately obvious by evaluating the estimator 
$$
\hat u_{\text{ls}} = (Y^TY)^{-1} Y^T Z^{-\frac 1 2} (D^T)^\dag b = A^{-1}D^T Z^{\frac 1 2} Z^{-\frac 1 2}(D^T)^\dag b = A^{-1}u = u^*. 
$$
Inspired by \cite{YB10} we consider projecting $u^*$ onto a low-dimensional subspace $\mathcal{S}_\rho$, spanned by a basis of $\rho$ linearly independent functions and thereafter attempt to simulate an approximate solution within this subspace. Given an $n \times \rho$ matrix $\Psi$ with $\rho < n$ orthonormal columns, the projection operator $\Pi\dot =\Psi \Psi^T$ maps vectors $u \in \mathbb{R}^n$ to the subspace
\be
\mathcal{S}_\rho \doteq \{\Psi r\, |\, r \in \mathbb{R}^\rho\},
\ee
such that for any $u = \Pi u + (I - \Pi) u$ there is a unique, optimal low-dimensional solution $r^*$ satisfying 
\be
\Psi r^* = \Pi u.
\ee 
Assigning $X=Y\Psi$, then if the basis $\Psi$ is chosen so that the projection error $(I - \Pi) u$ is sufficiently small, the task at hand is to evaluate a low-dimensional vector $r \in \mathbb{R}^\rho$ that approximates $r^*$ (respectively $\Psi r \in \mathbb{R}^n$ that approximates $\Pi u^*$), in the least squares sense 
\be\label{eqn:requiv}
r =\arg \min_{r\in \mathbb{R}^\rho} \bigl \|X r - (Y^T)^\dagger b \bigr \|^2,
\ee
whose solution is
\begin{align}\nonumber
r & =(X^TX)^{-1}X^T(Y^T)^\dagger b\\ \nonumber
& =(X^TX)^{-1}X^TY(Y^TY)^{-1} b\\ \nonumber
&=(X^TX)^{-1}\Psi^Tb\\ \nonumber
& =(\Psi^TA\Psi )^{-1}\Psi^TA u\\ \nonumber
& =(\Psi^TA\Psi )^{-1}\Psi^TA(\Pi u+(I-\Pi)u)\\ \label{eqn:r2formula}
&=\Psi^T u+(\Psi^TA\Psi )^{-1}\Psi^TA(I-\Pi)u.
\end{align}
Notice that, despite the reduction in the dimension of the solution, problem \eqref{eqn:requiv} turns out to be computationally more expensive than the original \eqref{aub} as it requires the pseudo-inverse of the large, parameter dependent $Y^T$ matrix. However, \eqref{eqn:requiv} admits a more efficient formulation, stated in the form of the following Lemma. 

\begin{lemma}
The solution of the least-squares problem \eqref{eqn:requiv} can be computed via the alternative formulation 
\begin{align}\label{eqn:requiv3}
r = \arg \min_{r\in \mathbb{R}^\rho}\bigl \|X r -Z^{-\frac{1}{2}}(D^T)^\dagger b \bigr \|^2.
\end{align}
\end{lemma}
\begin{proof} Developing the squared norm and introducing the expression of $(D^T)^\dagger$ we have 
\begin{align*}
(X^TX )^{-1}X^TZ^{-\frac{1}{2}}(D^T)^\dagger b &=(X^TX)^{-1}\Psi^TD^TZ^{\frac{1}{2}}Z^{-\frac{1}{2}}(D^T)^\dagger b\\
&=(X^TX)^{-1}\Psi^TD^T(D^T)^\dagger b\\ 
&=(X^TX)^{-1}\Psi^TD^TD(D^TD)^{-1} b\\
&=(X^TX)^{-1}\Psi^Tb\\
& =(\Psi^TA\Psi )^{-1}\Psi^TA u\\ &=(\Psi^TA\Psi )^{-1}\Psi^TA(\Pi u+(I-\Pi) u)\\
&=\Psi^T u + (\Psi^TA\Psi )^{-1}\Psi^TA(I-\Pi)u = r.
\end{align*}
\end{proof}
The fourth equation above indicates that for $X=Y\Psi$ the projected normal equations for the FEM system are\footnote{We emphasise the contrast between the projected equations in \eqref{eqn:requiv3normal} and the projected variable equations $\Psi^TA^TA \Psi r' = \Psi^T A^T b$ which correspond to the LS problem
$$
r' = \arg\min_{r \in \mathbb{R}^\rho} \bigl \|A \Psi r - b \bigr\|^2,
$$
the solution of which is
\begin{align*}\nonumber
r' & =(\Psi^TA^2\Psi )^{-1}\Psi^TA b\\ \nonumber
& =(\Psi^TA^2\Psi )^{-1}\Psi^TA^2 u\\ \nonumber
& =(\Psi^TA^2\Psi )^{-1}\Psi^TA^2(\Pi u+(I-\Pi) u)\\ 
& =\Psi^T u+(\Psi^TA^2\Psi )^{-1}\Psi^TA^2(I-\Pi) u,
\end{align*}
that incurs a subspace regression error term that is quadratic in $A$. Moreover, note that the right hand side vector in the normal equations $\Psi^TA^TA \Psi r' = \Psi^T A^T b$ has dependence on the parameter through $A$.}
\begin{align}\label{eqn:requiv3normal}
X^TXr = \Psi^T A \Psi r =\Psi^T b,
\end{align}
thus following up from the approach of Drineas et al. \cite{DrineasMahoneyER} we consider the randomisation of the projected coefficients matrix (the Hessian of the residual in \eqref{eqn:requiv3}) as in
\begin{align}\label{eqn:regrand3norm}
X^TSS^TX \hat r = \Psi^T b,
\end{align}
noticing that this can be deduced from \eqref{eqn:requiv3normal} 
$$
X^T SS^T X r + X^T (I - SS^T) X r = \Psi^T b,
$$
by neglecting the sketching error term $X^T (I - SS^T) X$. In the next sections we discuss how to randomise the computation of $X^T SS^T X$ using a sketching matrix $S$ that depends on the parameter vector $p$, while in the error analysis that follows we focus our attention on the various sources of errors affecting the induced sketched approximation of $u$ and in bounding the overall error.

So far we have discussed the projection of the high-dimensional system without providing explicit details on how the basis $\Psi$ is selected. A desired property of the appropriate basis is to sustain a small projection error $\|u - \Pi u\|$ for all admissible $p$ choices under the constraint $\rho \ll n$. Options include parameter-specific bases such as a subset of the right singular vectors of $A$ obtained through a randomised decomposition or Krylov-subspace bases which are orthogonalised via a Gram-Schmidt process \cite{Halko}. Here we opt for a generic basis exploiting the smoothness of the solution on Lipschitz domains. In particular, we select the basis among the eigenvectors of the discrete Laplacian operator
\be\label{Delta}
\Delta := D^TD,
\ee
for $D$ the gradients of the shape functions matrix in \eqref{Ydef} and $\Delta Q = Q \Lambda$, by splitting the eigenvectors $Q$ as
$$
Q=\bigl ( Q_{*(1:n-\rho-1)}| \Psi \bigr ),
$$
such that $\Psi$ corresponds to the last $\rho$ columns of $Q$ and the $\rho$ smallest eigenvalues $\{\lambda_{n-\rho-1}(\Delta), \ldots, \lambda_n(\Delta)\}$. This arrangement implies that the columns of $\Psi$ are ordered in decreasing spatial variation in $\Omega$
$$
\|D \Psi_{*(i)}\| > \|D \Psi_{*(j)}\| > 0, \quad \text{for} \quad \rho \geq i>j \geq 1.
$$
Clearly, the decomposition of $\Delta$ is computationally expensive so this can be performed offline, once, and then used the basis for all instances of the parameter vector. From \eqref{eqn:requiv3normal}, the existence of $r$ requires that $X^TSS^TX$ has full rank, hence it suffices to show that $SS^T \to I$ as the number of samples $c \rightarrow \infty$ with probability 1.

\section{Simulating the reduced system}\label{sec:RamSam}

In this section we focus attention to the randomised simulation of the reduced problem \eqref{eqn:regrand3norm}. In what follows we assume that all mesh-dependent quantities, including the basis $\Psi$ are readily available through offline computations, and are directly accessible from memory on demand. We aim to estimate the low-dimensional system matrix $\hat G:= X^TSS^TX$ in \eqref{eqn:regrand3norm} so that it maintains a minimal Frobenius norm from its deterministic counterpart
\eqref{eqn:requiv3normal}
\be
G \doteq X^TX = \sum_{\ell=1}^{kd} X^T_{(\ell)*} X_{(\ell)*} \;.
\ee
To do this assume a sampling distribution $\xi \doteq \{ \xi_\ell \}_{\ell=1}^{kd}$ with $\sum_{\ell=1}^{kd} \xi_\ell=1$ so that an index $\ell$ in the set $[kd]$ can be drawn with probability $\xi_\ell$. Then collecting $c \ll kd$ independent and identically distributed index samples $\{r_1, r_2, \ldots, r_c\}$ according to $\xi$ we can approximate $G$ as
\be\label{hatA}
\hat G \doteq X^TSS^T X = \frac 1 c \sum_{t=1}^c \frac{1}{\xi_{r_t}} X^T_{(r_t)*} X_{(r_t)*} 
\ee
for a sketching matrix $S=BC$, where $C$ is a $ c\times c$ diagonal matrix and $B$ is a tall $kd\times c$ sparse matrix with entries
\be \label{defBC}
C = \frac{1}{\sqrt{c}} \mathrm{diag} \bigl ( \xi^{-\frac12}_{r_1}, \ldots, \xi^{-\frac12}_{r_c} \bigr), \quad \text{and} \quad B = \bigl ( 1_{r_1}, \ldots, 1_{r_c}\bigr ),
\ee
and $1_i$ is the $i$th canonical vector. Indeed, $SS^T = BC^2B^T$ returns a $kd\times kd$ diagonal matrix with non-negative entries. It is important to observe that the above construction preserves the semi-definiteness and symmetry in $\hat G$, while involving significantly fewer operations compared to computing $G$. The estimator $\hat G$ can be shown to be an unbiased estimator of $G$ through probabilistic arguments.

\begin{prop}\label{prop:unbias} The matrix $\hat G$ constructed in \eqref{hatA} is an unbiased estimator for $G$ in the sense of $\mathbb{E}_{\xi}[\hat G]=G$. In effect, when $c\to \infty$, $\mathrm{Var}_{\xi}[\|G-\hat G\|_F]\to 0$ with probability 1.
\end{prop}
 
\begin{corollary}\label{col:unbias} Define $S=BC$. Then $SS^T$ is an unbiased estimator for $I$ the identity matrix, i.e. $\mathbb{E}_{\xi}[SS^T]=I$ under probability $\xi$.
\end{corollary}

An optimal choice for $\xi_\ell$ in the sense of minimising the expectation of the Frobenius norm of the so-called simulation error $G - \hat G$, can be made according to the parameter vector $p$ as shown next.
\begin{prop}\label{prop:minsigma} The optimal sampling probability $\xi$ for $\hat G$ in \eqref{hatA} in the sense of minimising the error $\mathbb{E}_{\xi}[\|G-\hat G\|^2_F]$ is given by 
\be\label{optimalmatrix}
\xi_\ell =\frac{\|X_{(\ell)*}\|^2}{\|X\|_F^2},  \ \ \text{for all}\ \ 1\leq \ell\leq kd,
\ee
for which the corresponding variance is bounded by
\be\label{eqn:optimalvar} 
\mathrm{Var}_{\xi}[\|G-\hat G\|_F]\leq \mathbb{E}_{\xi}[\|G-\hat G\|^2_F] \leq \frac{1}{c} \Bigl(\sum_{\ell=1}^{kd} \|X_{(\ell)*}\|\Bigr)^2 \leq \frac{d^2}{c} p_\Omega^2 \|D\|^2,  
\ee 
where $p_\Omega =  \sum_{\ell=1}^k p_\ell |\Omega_\ell|$ is the discretised integral of the parameter function $p(x)$ over the domain.
\end{prop}

Note that for an arbitrary sampling distribution $\xi$, the singular values of $\hat G$ can be shown to be bounded by the product $d\, p_\Omega$ and further bounded in terms of the sample budget $c$ and the corresponding singular values of $G$. Since $X=Y\Psi = Z^{\frac 1 2} D \Psi$, then at a fixed $\mathcal{T}_\Omega$, the norms of the rows of $D\Psi$ can be computed offline, allowing the distribution to be swiftly computed by scaling as $\xi_\ell \propto Z_{\ell\ell}^{\frac 1 2}\|(D\Psi)_{(\ell)*}\|$. 

\begin{prop}\label{prop:max}
Assume the randomised sampling procedure in \eqref{hatA} for approximating $G$ with sampling probabilities as in \eqref{optimalmatrix}, then the spectrum of $\hat G$ is bounded from above as
$$
\sigma_1(\hat G)\leq \sum_{\ell=1}^{kd} z_\ell \|(D\Psi)_{(\ell)*} \|^2 \leq d\, p_\Omega \|D\|^2.
$$
\end{prop}

Complementary to Proposition \ref{prop:max}, the positive singular values of $\hat G$ can be bounded by the corresponding singular values of $G$ and its minimum singular value.
\begin{prop}\label{prop:min}
Assume that $\hat G$ is full rank, then for any $\gamma\in (0,1)$ 
\begin{align}
\sigma_{i}(\hat G )\geq\sigma_{i}(G )-\gamma\sigma_{\min}( G) \quad \text{for} \quad 1\leq i\leq \rho,
\end{align}
holds with probability at least 
$$
1-\min\Big\{1,\frac{\mathbb{E}_{\xi}[\|G-\hat G\|_F]}{\gamma\sigma_{\min}( G)}\Big\}.
$$ 
\end{prop}

Due to the positive semi-definiteness of $G $ and $\hat G $, their singular values coincide with their eigenvalues, while for $c\ll \rho$ then almost surely $\hat G$ is full rank. These results, in conjunction with Lemmas \ref{largesteig} and \ref{smallesteig} will be used in calculating the simulation error in Section \ref{sec:ErrAna}. The total computational cost for obtaining \eqref{hatA} is at most $\mathcal{O}(c\rho^2)+\mathcal{O}(k)$.

\begin{algorithm}
\caption{Randomised simulation algorithm}\label{alg:sketch1}
\begin{algorithmic}[1] 
\State \textbf{Input}: Matrix $D\Psi \in \mathbb{R}^{kd \times \rho}$ (offline), vector $\Psi^Tb \in \mathbb{R}^\rho$ (offline), and element volumes vector $\omega \in \mathbb{R}^k$
\For{$i=1,2,\ldots,N$}
\State \textbf{input} parameters vector $p \in \mathbb{R}^k$
\State Compute vector $z = p \odot \omega$ and diagonal $Z = z \otimes I_d$ 
\State Compute $\xi_\ell \propto Z_{\ell\ell}^{\frac 1 2}\|(D\Psi)_{(\ell)*}\|$
\State Scale the rows of $D\Psi$ to get $X=Z^{\frac 1 2}(D\Psi)$
\State Draw $c$ iid samples from $\xi$ to assemble $S$
\State Compute $\hat G =  X^TSS^TX$
\State \textbf{Output}: $\hat r = \hat G^{-1} (\Psi^T b)$ and $\hat u = \Psi \hat r$.
\EndFor\label{euclidendwhile}
\State \textbf{end}
\end{algorithmic}
\end{algorithm}

\subsection{Statistical leverage score sampling}

As proved in \cite{DrineasMahoneyER} for the graph Laplacian paradigm and later in \cite{AvronToledo} for the FEM stiffness matrix, optimal sampling probabilities for regression problems are derived based on statistical leverage scores. As these scores are typically impractical to compute, it is reasonable to consider approximating them in a computationally efficient way. In doing so we first investigate the discrepancy between the leverage scores probability and that used in our algorithm, in our setting. We argue that the subspace projection causes the distance between the two to reduce, and thus there is a significant performance advantage in simulating the product $G = (Y \Psi)^T(Y \Psi)$ instead of $A=Y^TY$ when $\rho < n$.
To show this, consider that for a matrix $B$ with $kd$ rows we can define the statistical leverage score and row norm sampling probabilities as
\be
\xi^{l(B)}=l(B)/\sum_{\ell=1}^{kd} l_\ell(B)\ \ \ \text{and}\ \ \  \xi^{r(B)}=r(B)/\sum_{\ell=1}^{kd} r_\ell(B), 
\ee
respectively, where the $\ell$th leverage score and row-norm squared for $B=U_B\Sigma_B V_B^T$ are
\be
l_\ell(B):= (U_B U_B^T)_{\ell \ell} \quad \text{and} \quad r_\ell(B):=\|B_{\ell*}\|^2=(BB^T)_{\ell\ell}
\ee
with $\ell=1,\ldots, kd$. From this we seek to show that the  projection onto the low-dimensional subspace induces the inequalities 
\begin{align}\label{conjection1}
\|\xi^{l(X)} -  \xi^{r(X)}\|_{\text{norm}} \leq \|\xi^{l(Y)}-\xi^{r(Y)}\|_{\text{norm}}, 
\end{align}
where $\|\cdot\|_{\text{norm}}$ can be either $\|\cdot\|$ or $\|\cdot\|_{\max}$.
For clarity we address first the simple case where $Z$ is uniform, that is when both $p$ and $\omega$ are uniform vectors. 

\subsection{Simple case: homogeneous model} 

For the $kd \times n$ matrix $D$ in $Y = Z^{\frac 1 2}D$ where $kd > n$ we have $D = U_D \Sigma_D V_D^T$ where $U_D \in \mathbb{R}^{kd \times n}$ and $\Sigma_D \in \mathbb{R}^{n \times n}$ is a nonzero diagonal whose values are denoted by $\lambda_1(\Sigma_D)\leq \lambda_2(\Sigma_D)\leq \cdots \leq \lambda_n(\Sigma_D)$.
\begin{lemma}\label{lemma:Xdist}
In the homogeneous model $Z=zI$ with $z>0$, we have that
\begin{align}\label{eqn:Xhomoinq}
    \| \xi^{r(X)}-\xi^{l(X)}\|_{\max}\leq \Big(\frac{\lambda_{n-\rho+1}(\Sigma_D)^2}{\sum_{i=n-\rho+1}^n \lambda_{i}(\Sigma_D)^2}-\frac{1}{\rho}\Big)\vee\Big(\frac{1}{\rho}-\frac{\lambda_{n}(\Sigma_D)^2}{\sum_{i=n-\rho+1}^n \lambda_{i}(\Sigma_D)^2}\Big),
    \end{align}
    and
    \begin{align}\label{eqn:Xhomoinq2}
   & \| \xi^{r(X)}-\xi^{l(X)}\|\leq\sqrt{\sum_{i=n-\rho+1}^{n}\Big(\frac{\lambda_{i}(\Sigma_D)^2}{\sum_{i=n-\rho+1}^n \lambda_{i}(\Sigma_D)^2}-\frac{1}{\rho}\Big)^2}.
\end{align}
\end{lemma}
\begin{proof}
From the SVD of $D$ that of the discrete Laplacian is $\Delta = D^TD = V_D \Sigma_D^2 V_D^T$ and we can form the $n \times \rho$ basis $\Psi$ by partitioning as
\be\label{eqn:partition}
\Sigma_D = \begin{pmatrix} \bar \Sigma_D & 0\\ 0 & \bar \Sigma_\rho \end{pmatrix}, \quad \text{and} \quad V_D = ( \bar V_D | \Psi ),
\ee
where $\bar \Sigma_\rho$ is $\rho \times \rho$, and clearly $\mathrm{Trace}(\bar \Sigma_D) > \mathrm{Trace}(\bar \Sigma_\rho)$. We can now write the decomposition $Y = \sqrt{z} U_D \Sigma_D V_D^T$ and thereafter
$$
X = \sqrt{z} U_D \Sigma_D \begin{pmatrix} 0\\ I_\rho \end{pmatrix} = \sqrt{z} U_D \begin{pmatrix} 0\\ \bar \Sigma_\rho  \end{pmatrix}= \sqrt{z}(U_D)_{*(n-\rho+1:n)}\bar \Sigma_\rho,
$$
where $(U_D)_{*(n-\rho+1:n)}$ is the submatrix of $U_D$ from column $n-\rho+1$ to $n$. We can now express the leverage scores for $X$ as 
\be\label{eqn:ellX}
l_i(X) = \mathrm{diag}\big((U_D)_{*(n-\rho+1:n)}(U_D)_{*(n-\rho+1:n)}^T\big)_i,
\ee
and the probabilities associated with leverage scores of $X$ as
$$\xi^{l(X)}=\frac{1}{\rho}\mathrm{diag}\big((U_D)_{*(n-\rho+1:n)}(U_D)_{*(n-\rho+1:n)}^T\big).$$
Similarly,
\be\label{eqn:rX}
r_i(X) = z\mathrm{diag}\Big((U_D)_{*(n-\rho+1:n)}  \bar \Sigma_\rho^2  (U_D)_{*(n-\rho+1:n)}^T\Big)_i,
\ee
with associated probabilities 
$$\xi^{r(X)}=\frac{1}{\text{Trace}(\bar \Sigma_\rho^2)}\mathrm{diag}\Big((U_D)_{*(n-\rho+1:n)}  \bar \Sigma_\rho^2  (U_D)_{*(n-\rho+1:n)}^T\Big).$$
It is now apparent that
\begin{align}\label{eqn:Xdisthomolemma}
\begin{split} \| \xi^{r(X)}-\xi^{l(X)}\|_{\max} & =\Big\|\mathrm{diag}\Big((U_D)_{*(n-\rho+1:n)} \Big(\frac{1}{\text{Trace}(\bar \Sigma_\rho^2)} \bar \Sigma_\rho^2 -\frac{1}{\rho} I_{\rho}\Big) (U_D)_{*(n-\rho+1:n)}^T\Big)\Big\|_{\max}\\
 &\leq \Big\|(U_D)_{*(n-\rho+1:n)} \Big(\frac{1}{\text{Trace}(\bar \Sigma_\rho^2)} \bar \Sigma_\rho^2 -\frac{1}{\rho} I_{\rho}\Big) (U_D)_{*(n-\rho+1:n)}^T\Big\|\\
 &=\Big\|\frac{1}{\text{Trace}(\bar \Sigma_\rho^2)} \bar \Sigma_\rho^2 -\frac{1}{\rho} I_{\rho}\Big\|\\
 &=\Big(\frac{\lambda_{n-\rho+1}(\Sigma_D)^2}{\sum_{i=n-\rho+1}^n \lambda_{i}(\Sigma_D)^2}-\frac{1}{\rho}\Big)\vee\Big(\frac{1}{\rho}-\frac{\lambda_{n}(\Sigma_D)^2}{\sum_{i=n-\rho+1}^n \lambda_{i}(\Sigma_D)^2}\Big).
 \end{split}
\end{align}
Alternatively, taking the Euclidean norm gives
\begin{align*}
    \| \xi^{r(X)}-\xi^{l(X)}\| & = \| \xi^{r(X)}-\xi^{l(X)}\|_F \\
&\leq \Big\|(U_D)_{*(n-\rho+1:n)} \Big(\frac{1}{\text{Trace}(\bar \Sigma_\rho^2)} \bar \Sigma_\rho^2 -\frac{1}{\rho} I_{\rho}\Big) (U_D)_{*(n-\rho+1:n)}^T\Big\|_{F}\\
&\leq \Big\|\frac{1}{\text{Trace}(\bar \Sigma_\rho^2)} \bar \Sigma_\rho^2 -\frac{1}{\rho} I_{\rho} \Big\|_{F}\\
&=\sqrt{\sum_{i=n-\rho+1}^{n}\Big(\frac{\lambda_{i}(\Sigma_D)^2}{\sum_{j=n-\rho+1}^n \lambda_{j}(\Sigma_D)^2}-\frac{1}{\rho}\Big)^2}.
\end{align*}
\end{proof}
\begin{remark}
If we define $\zeta_j:=\frac{\lambda_{j}(\Sigma_D)^2}{\sum_{i=n-\rho+1}^n \lambda_{i}(\Sigma_D)^2}$, then the upper bound of $\| \xi^{r(X)}-\xi^{l(X)}\|$, as shown in Lemma \ref{lemma:Xdist}, characterises the discrepancy between $\zeta$ and the uniform probability, while the upper bound of $\| \xi^{r(X)}-\xi^{l(X)}\|_{\max}$ measures the largest deviation of $\zeta$ from the uniform probability. 
\end{remark}
The next result states the existence of a suitable $\rho$ such that the difference between sampling probabilities is smaller after projection with $\rho$ bases. 
\begin{corollary} \label{conjecture1_hom}
There exists at least one $\rho \in [n]$ such that \eqref{conjection1} holds.
\end{corollary}
\begin{proof}
First, similar as in \eqref{eqn:Xdisthomolemma}, we can conclude for $ \| \xi^{r(Y)}-\xi^{l(Y)}\|_{\max}$ that
\begin{align}\label{eqn:XdisthomolemmaY}
\begin{split} 
\| \xi^{r(Y)}-\xi^{l(Y)}\|_{\max} & =\Big\|\mathrm{diag}\Big(U_D \Big(\frac{1}{\text{Trace}(\bar \Sigma_\rho^2)} \bar \Sigma_\rho^2 -\frac{1}{\rho} I_{\rho}\Big) (U_D)^T\Big)\Big\|_{\max}\\
 &\geq \frac{1}{n}\Big\|U_D \Big(\frac{1}{\text{Trace}(\bar \Sigma_\rho^2)} \bar \Sigma_\rho^2 -\frac{1}{\rho} I_{\rho}\Big) (U_D)^T\Big\|\\
 &=\frac{1}{n}\Big\|\frac{1}{\text{Trace}(\bar \Sigma_\rho^2)} \bar \Sigma_\rho^2 -\frac{1}{\rho} I_{\rho}\Big\|\\
 &=\frac{1}{n}\Big(\frac{\lambda_{1}(\Sigma_D)^2}{\sum_{i=1}^n \lambda_{i}(\Sigma_D)^2}-\frac{1}{n}\Big)\vee\Big(\frac{1}{n}-\frac{\lambda_{n}(\Sigma_D)^2}{\sum_{i=1}^n \lambda_{i}(\Sigma_D)^2}\Big).
 \end{split}
\end{align}
To achieve \eqref{conjection1}, we need to find a $\rho\in [n]$ such that the upper bound obtained in
\eqref{eqn:Xinhomoinq} is no bigger than the lower bound above, i.e., 
\begin{align}\label{conjecture1_bounds}
\begin{split}
&\frac{1}{n}\Big(\frac{\lambda_{1}(\Sigma_D)^2}{\sum_{i=1}^n \lambda_{i}(\Sigma_D)^2}-\frac{1}{n}\Big)\vee\Big(\frac{1}{n}-\frac{\lambda_{n}(\Sigma_D)^2}{\sum_{i=1}^n \lambda_{i}(\Sigma_D)^2}\Big)\\
& \geq \Big(\frac{\lambda_{n-\rho+1}(\Sigma_D)^2}{\sum_{i=n-\rho+1}^n \lambda_{i}(\Sigma_D)^2}-\frac{1}{\rho}\Big)\vee\Big(\frac{1}{\rho}-\frac{\lambda_{n}(\Sigma_D)^2}{\sum_{i=n-\rho+1}^n \lambda_{i}(\Sigma_D)^2}\Big).
\end{split}
\end{align}
This can be easily verified by setting $\rho=1$. 
\end{proof}
\begin{remark} $\rho$ is not necessarily to be $1$. Indeed \eqref{conjecture1_bounds} gives a rather strict bound which may narrow the choice for $\rho$. On the other hand, consider that we find a range of $\rho$ such that
\begin{align}\label{eqn:condhomlemma}
F(\rho)\dot{=}\frac{\lambda_{n-\rho+1}(\Sigma_D)^2}{\sum_{i=n-\rho+1}^n \lambda_{i}(\Sigma_D)^2}-\frac{1}{\rho}\geq \frac{1}{\rho}.
\end{align}
In effect, in this range the upper bound of $\|\xi^{l(X)} - \xi^{r(X)}\|_{\max}$ in \eqref{eqn:Xhomoinq} is indeed $F(\rho)$.
%
In general, we expect this range to include large (integer) values close to and equal to $n$. Now seek the smallest $\rho$ such that both $F(\rho)\leq F(n)$ and \eqref{eqn:condhomlemma} hold. Roughly speaking, it is highly likely for \eqref{conjection1} to be true for this particular $\rho$ as the corresponding upper bounds have the relation $F(\rho)\leq F(n)$.
\end{remark}

\section{General case: inhomogeneous model}

Typically, the parameter vector $p$ and the element volumes $\omega$ have arbitrary positive values, thus $Z \neq zI$. Here we adapt the homogeneous model analysis to investigate whether the effect of the projection on the sampling probabilities is sustained in this case too. 

\begin{lemma}\label{lem:inhom}
In the inhomogeneous model, we have 
\begin{align}\label{eqn:Xinhomoinq}
    \| \xi^{r(X)}-\xi^{l(X)}\|_{\max}\leq \big(\max_i \pi_i(\rho)-\frac{1}{\rho}\big)\vee \big(\frac{1}{\rho}-\min_i \pi_i(\rho)\big),
    \end{align}
    and
    \begin{align}\label{eqn:Xinhomoinq2}
   & \| \xi^{r(X)}-\xi^{l(X)}\|\leq\sqrt{\sum_{i=1}^\rho \pi_i(\rho)^2-\frac{1}{\rho}},
\end{align}
where $\pi_i(\rho):=\frac{  \lambda_i(\Sigma_X^2)}{\|X\|_F^2}$ for $i\in [\rho]$. In addition, there exists at least one $\rho\in [n]$ such that \eqref{conjection1} holds.
\end{lemma}

\begin{remark}
\begin{enumerate}
\item Like the homogeneous case (see Lemma \ref{lemma:Xdist}), the upper bound of $\|\xi^{l(X)} - \xi^{r(X)}\|$ in \eqref{eqn:Xinhomoinq2} characterises the discrepancy between the probability $\pi(\rho)$ and the uniform probability, while the upper bound of $\|\xi^{l(X)} - \xi^{r(X)}\|_{\max}$ in \eqref{eqn:Xinhomoinq} is measured by the largest deviation of the probability $\pi(\rho)$ from the uniform probability.
\item 
Though there is no clear evidence for an increasing trend for  $\|\xi^{l(X)} - \xi^{r(X)}\|$ with respect to $\rho$, not even for the upper bound of it in \eqref{eqn:Xinhomoinq}, the numerical experiments on the discrepancy between the sampling distributions for $X$ and $Y$ presented in figure \ref{xicomp} illustrate roughly this trend. This plot shows that the smaller $\rho$ we choose, the less the difference between row-sampling and statistical leverage sampling is. 
\end{enumerate}
\end{remark}

\begin{corollary}
For an arbitrary matrix $Y$ of size $kd\times n$ where $kd > n$ and $Y$ has rank $n$, we have
\begin{align}\label{eqn:Yinhomoinq}
    \| \xi^{r(Y)}-\xi^{l(Y)}\|_{\max}\leq \Big(\frac{\|Y\|^2}{\|Y\|^2_F}-\frac{1}{\rho}\big)\vee \big(\frac{1}{\rho}-\frac{\lambda_n(\Sigma_Y)^2}{\|Y\|_F^2}\Big),
    \end{align}
    and
    \begin{align}\label{eqn:Yinhomoinq2}
   & \| \xi^{r(Y)}-\xi^{l(Y)}\|\leq\sqrt{\sum_{i=1}^\rho \Big(
   \frac{\lambda_i(\Sigma_Y)^2}{\sum_{j=1}^{n}\lambda_j(\Sigma_Y)^2}-\frac{1}{\rho}\Big)^2}.
\end{align}
\end{corollary}
This corollary is a consequence of Lemma \ref{lem:inhom} with $\Psi=I$. The result shows that the difference of choosing between the two sampling probabilities is mainly determined by the dispersion in the singular values of $Y$.


\subsection{Error Analysis}\label{sec:ErrAna}

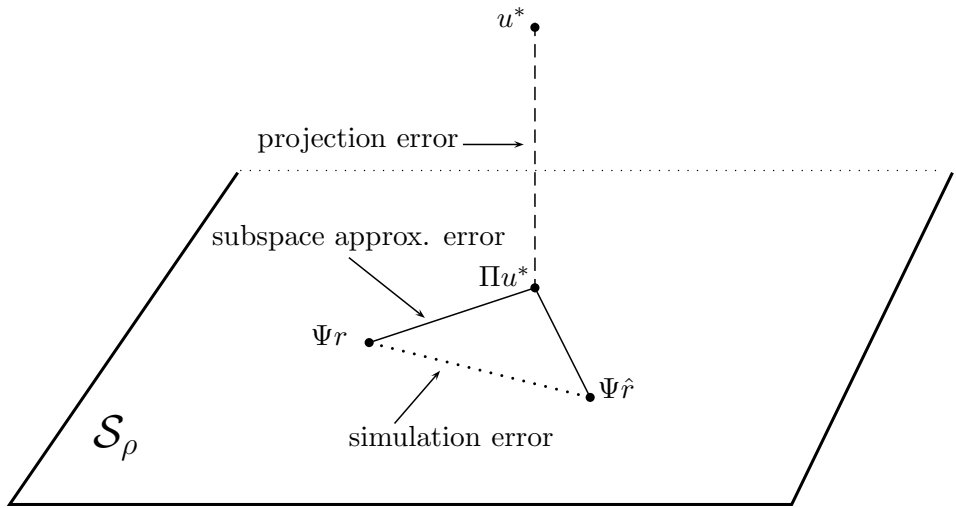
\begin{figure}
\psscalebox{1.0 1.0}
{
\begin{pspicture}(0,-3.2936363)(12.522699,3.2936363)
\psline[linecolor=black, linewidth=0.04](3.0378351,1.1263636)(0.037835084,-3.2736363)(10.22355,-3.2736363)(10.323549,-3.2736363)(12.437835,1.1263636)
\rput[b](1.437835,-2.6869698){\LARGE $\mathcal{S}_\rho$}
\rput[bl](6.446835,3.0536363){$u^*$}
\rput[bl](6.210562,-0.3809091){$\Pi u^*$}
\rput[bl](4.0014715,-1.1645454){$\Psi r$}
\psline[linecolor=black, linewidth=0.02](4.7651076,-1.1281818)(6.946926,-0.4009091)
\psline[linecolor=black, linewidth=0.04, linestyle=dotted, dotsep=0.10583334cm](4.7651076,-1.1281818)(7.6741986,-1.8554546)
\rput[bl](7.7741986,-1.8554546){$\Psi \hat r$}
\psdots[linecolor=black, dotsize=0.12](6.946926,3.0536363)
\psdots[linecolor=black, dotsize=0.12](6.946926,-0.4009091)
\psdots[linecolor=black, dotsize=0.12](4.7651076,-1.1281818)
\psdots[linecolor=black, dotsize=0.12](7.6741986,-1.8554546)
\psline[linecolor=black, linewidth=0.02, linestyle=dotted](3.0681381,1.1548484)(12.322683,1.1548484)
\psline[linecolor=black, linewidth=0.02, linestyle=dashed, dash=0.17638889cm 0.10583334cm](6.946926,3.1142423)(6.946926,-0.46)
\psline[arrows=->,linewidth=0.02,linestyle=solid]{->}(6,1.5) (6.8,1.5)
\rput[bl](3.3,1.4){projection error}
\psline[arrows=->,linewidth=0.02,linestyle=solid]{->}(4.5,0) (5.5,-0.8)
\rput[bl](2.7,0.1){subspace approx. error}
\rput[bl](4.5,-2.5){simulation error}
\psline[arrows=->,linewidth=0.02,linestyle=solid]{->}(5.1,-2.2)(5.7,-1.5)
\psline[linewidth=0.02, linestyle=solid, dash=0.17638889cm 0.10583334cm](6.946926,-0.4009091)(7.6741986,-1.8554546)
\end{pspicture}
}
\label{schem}
\caption{A geometric interpretation of the error components imparted in the sketched solution $\Psi \hat r$. Starting from the high-dimensional, `exact' FEM solution $u^* = A^{-1}b$ we project orthogonally onto the subspace $\mathcal{S}_\rho$ arriving at $\Pi u^*$ while incurring some projection error. The projected problem then leads to a low-dimensional solution $\Psi r = \Psi G^{-1} \Psi^Tb$ that in turn incurs a subspace approximation error due to the condition of the projected matrix $G$, and ultimately $\Psi r$ is approximated via its sketched version $\Psi \hat r = \Psi \hat{G}^{-1} \Psi^T b$ that includes simulation error due to the variance in the estimated $\hat G$.}
\end{figure}

Our approach for simulating a projected solution to the FEM system contends with various sources of error. As depicted at the schematic in figure \ref{schem}, there is an approximation error component associated with restricting to the subspace $\mathcal{S}_\rho$, and this error can be further decomposed in two parts, the projection error, given by $\|u-\Pi u\|$, measuring the distance between the exact solution $u$ and its projection onto the subspace $\Pi u$; and the subspace approximation error, given by $\|\Pi u- \Psi r\|$, measuring the distance between the projection of the true solution $u$ to the best approximation of the regression point $r$ from \eqref{eqn:requiv} within $\mathcal{S}_\rho$. In addition, there is also a simulation or sketching error $\|\Psi r - \Psi \hat r\|$ associated with solving the projected problem based on the sketched $\hat G$ instead of the deterministic $G$. It is easy to show that the distance between $\Pi u$ and $\Psi r$ can be bounded in terms of approximation error. 

\begin{prop}\label{prop:SubReg_second} Considering the regression problem \eqref{eqn:requiv} and recalling that $G=\Psi^TA\Psi$, then we have 
\begin{align*} 
\|\Pi u - \Psi r \|\leq  \frac{\lambda_{\max}(A)}{\lambda_{\min}(G)}\|u-\Pi u\|
\end{align*}
where $\lambda_{\max}(A) =\lambda_{\max}(\Sigma_Y)^2$ is bounded in Lemma \ref{largesteig}, and $\lambda_{\min}(G) = \lambda_\rho(\Sigma_X)^2$.
\end{prop}
\begin{proof}
From the expression for $r$ in formula \eqref{eqn:r2formula} we immediately obtain
\begin{align*}
    \|\Psi r - \Pi u\|=\|\Psi (\Psi^TA\Psi)^{-1}\Psi^T A(I-\Pi)u\|\leq \frac{\lambda_{\max}(A)}{\lambda_{\min}(G)}\|u-\Pi u\|.
\end{align*}
\end{proof}

The simulation error associated with replacing $\Psi r$ by the sketching-based approximation $\Psi\hat r$ from \eqref{eqn:regrand3norm}, is given by $\|\Psi r - \Psi \hat r\|$, and can be bounded in terms of $\|\Psi r\|$.

\begin{prop}\label{prop:SimErr3-second} Assume problem settings as discussed in sections \ref{sec:DimRed} and \ref{sec:RamSam}, and consider the sketched system in \eqref{eqn:regrand3norm}. Then any $\epsilon,\delta\in (0,1)$ and $c$ chosen as
\begin{align}\label{eqn:cbounds-second}
(1+\frac 1 \epsilon)^2\frac{\bigl((\sum_{\ell=1}^{kd} z_\ell \|(D\Psi)_{(\ell)*}\|)^2- \|G\|^2_F  \bigr)}{\lambda_{\min}(G)^2 \, \delta}\leq c
\end{align}
satisfy 
\be \label{eqn:errorterms-second}
\|\Psi r - \Psi \hat{r}\| \leq \epsilon\|\Psi r\|
\ee
with probability $1-\delta$.
\end{prop}

An application of Proposition \ref{prop:SubReg_second} and Proposition \ref{prop:SimErr3-second} leads to the following result.
\begin{theorem}\label{thm:totalerror3_second}Assume the settings as in Proposition \ref{prop:SimErr3-second} with $\epsilon$, $\delta$ and $c$. Then 
\begin{align}\label{eqn:totalerror3_second} 
\|\Psi \hat r - u\| \leq   & \frac{\lambda_{\max}(A)}{\lambda_{\min}(G)}\|u-\Pi u\|+\epsilon\|\Psi r\|,
\end{align}
with probability $1-\delta$.
\end{theorem}

\begin{remark} From Cauchy's interlacing theorem \cite{Woodruff} and $G=\Psi^TA \Psi$ we have
$$
 \lambda_{\min}(A)\leq  \lambda_{\min}(G)\leq  \lambda_{\rho}(A),
 $$
thus the best to be expected from \eqref{eqn:totalerror3_second} is
$$
\|\Psi \hat r -u\|\leq \kappa_{\rho}(A)\|u-\Pi u\|+\epsilon\|\Psi r\|,
$$
where $\kappa_{\rho}(A):=\frac{\|A\|_2}{\lambda_{\rho}(A)}$ with $c$ chosen to as in Proposition \ref{prop:SimErr3-second}.
On the other hand the worst case is
$$\|\Psi \hat r -u\|\leq \kappa(A)\|u-\Pi u\|+\epsilon\|\Psi r\|.$$
\end{remark}

We end the error analysis by making a remark on the condition number of matrix $G$, noticing that from Theorem \ref{thm:totalerror3_second} the total error is bounded by $\|(I-\Pi)u\|$ and $\|\Psi r\|$.
\begin{remark}\label{kappaG}
The quantity $\|\Psi r\|$ can be further developed following the last line of \eqref{eqn:r2formula} as
\begin{align*}
\|\Psi r\| & \leq \|\Psi \Psi^Tu\| + \|\Psi(\Psi^TA\Psi)^{-1} \Psi^T A(I- \Pi)u\|\\
& \leq \|\Pi u\| + \bigl (\kappa(G) + \frac{\lambda_{\max}(A) - \lambda_{\max}(G)}{\lambda_{\min}(G)} \bigr ) \|(I - \Pi)u\|,
\end{align*}
therefore when $\kappa(G)$ is large, $\|\Psi r\|$ and in turn the total error increases.
\end{remark}

\section{Numerical experiments}\label{sec:ne}

To verify the performance of our algorithm and to test the derived error bounds we designed a number of numerical experiments based on the Dirichlet and Neumann problems \eqref{pde}-\eqref{bnd}. In these we consider a domain $\Omega$ to be a sphere of unit radius centred at the origin and discretised into $k=190955$ unstructured linear tetrahedral elements. This model comprises a total $n+n_\partial= 34049$ nodes of which $n_\partial = 4217$ are on the boundary. In the tests discussed below we run a sequence of $N=1000$ FEM problems where $p$ is chosen at random, and present our findings on average for the thousand problems. For the subspace projection, a basis $\Psi$ consisting of singular vectors of $D^TD$ was used throughout. Ahead of the tests we compute and store the mesh-dependent sparse gradients matrix $D$ modified to conform to the imposed boundary conditions, and the tall matrix $D\Psi$. Effectively, given $p$ one readily forms the diagonal $Z$ and thereafter the solution is computed directly as \verb"u=A\b" once the stiffness matrix \verb"A=D'ZD" is assembled. Our code was implemented in Matlab R2018b and executed on a workstation equipped with two 14-core Intel Xeon dual processors, running Linux NixOS with 384GB RAM.  

\begin{figure}
\begin{center}
\includegraphics[width=0.49\textwidth]{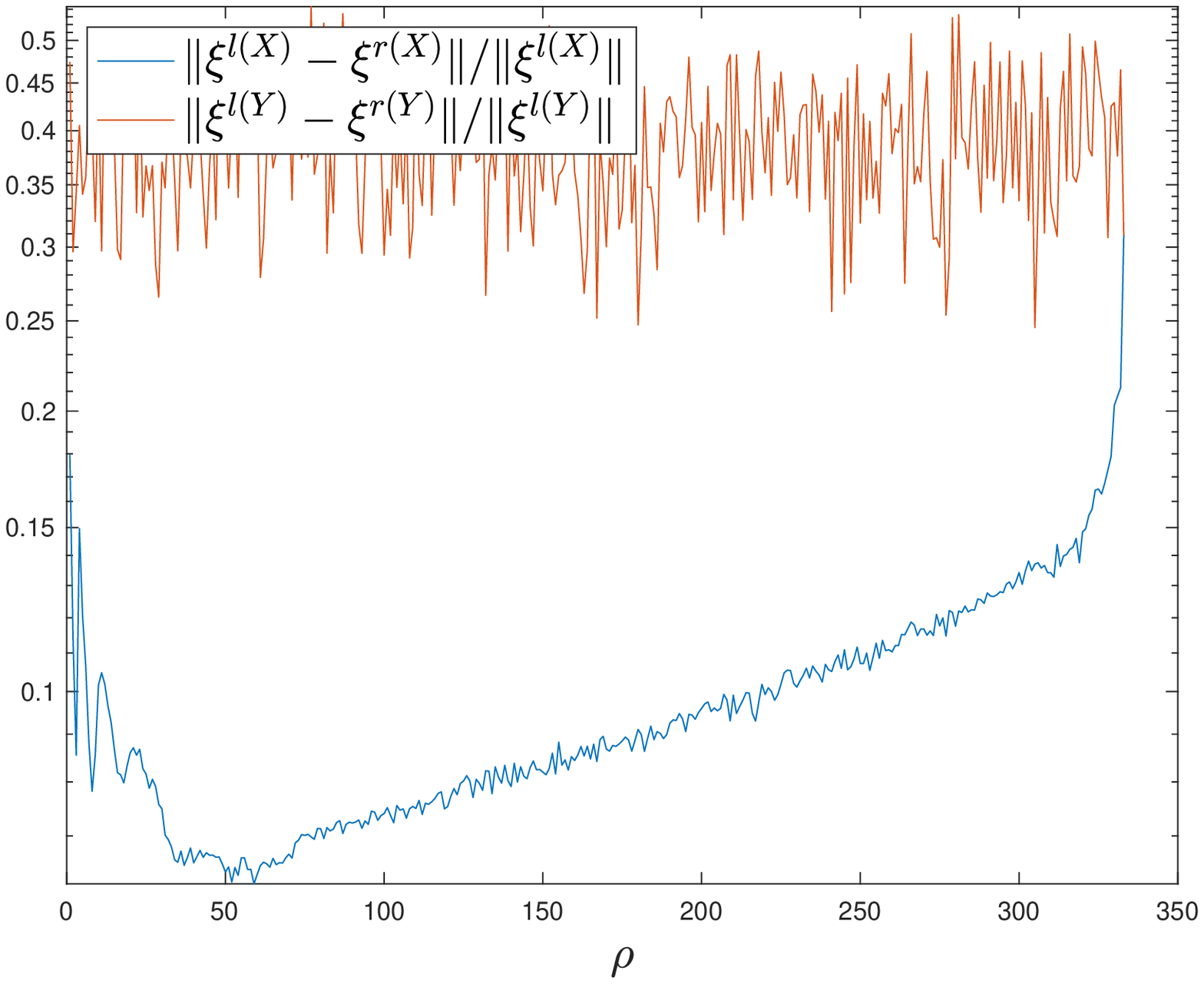} \includegraphics[width=0.49\textwidth]{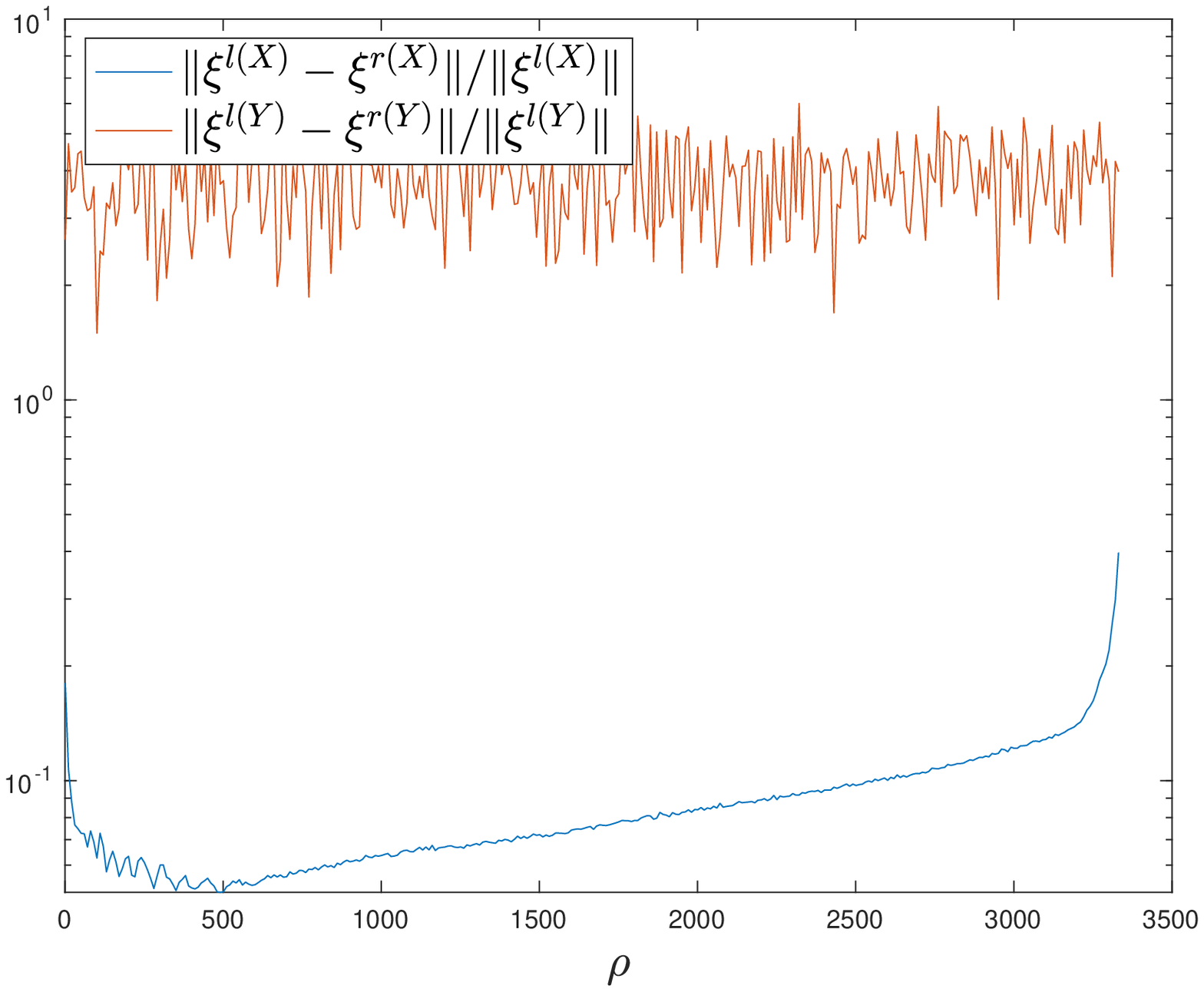}\\
\caption{Numerical investigation of the distance between the utilised and leverage score based sampling distributions for the projected $G=X^TX$ (blue) and the original $A=Y^TY$ (red) matrix products for varying $\rho$ on two coarser meshes of the domain with $n=334$ (left) and $n=3335$ (right) degrees of freedom respectively. The optimal sampling distributions $\xi^{l(X)}$ and $\xi^{l(Y)}$ are taken to be those based on the statistical leverage scores as in \cite{DrineasMahoneyER} of $X$ ans $Y$ respectively, while the $\xi^{r(X)}$ and $\xi^{r(Y)}$ are those implemented in our algorithm and are based on the Euclidean norm of the matrix rows. Note that the plots are averaged over 5 $(X,Y)$ pairs involving randomly drawn vectors $p$ from the uniform distribution $\mathcal{U}[10^{-2},1]$. The graphs show explicitly that sketching the projected product $X^TX$ for $\rho \ll n$ with $\xi^{r(X)}$ is near optimal, as well as illustrating the range of $\rho$ values where the discrepancy between the two distributions exhibits a monotonic behaviour.}
\label{xicomp}
\end{center}
\end{figure}

\subsection{The Dirichlet problem}

\begin{table}
\centering
\begin{tabular}{|c|c|c|c|c|c|c|c|c|c|}
  \hline
test & $\rho$ & $c$ & time & ratio & $\frac{\|\Pi u - u\|}{\|u\|}$  & $\frac{\|\hat G - G\|_F}{\|G\|_F}$ & $\kappa(G)$ & $\frac{\|\hat r - r\|}{\|r\|}$ & $\frac{\|\Psi \hat r - u\|}{\|u\|}$ \\ \hline \hline
A &  100  & 5000 & 658 & 0.0087 & 0.0420 & 0.1312 & 16.3& 0.0796 & 0.0914\\ 
B  & 50  & 5000 & 609 &  0.0087 & 0.0675 & 0.1309 & 10.9 & 0.0783 & 0.0913\\
C &  50  & 10000 & 396 & 0.0087 & 0.0675 & 0.0924 & 10.2 & 0.0624 & 0.0992\\
D &  50  & 10000 & 448 & 0.0087 & 0.0662 & 0.0923 & 10.9 & 0.0613 & 0.0942\\
E &  50  & 50000 & 495 & 0.0806 & 0.0675 & 0.0292 & 10.5 & 0.0193 & 0.0861\\
F &  50  & 100000 & 574 & 0.1496 & 0.0675 & 0.0207 & 10.8 & 0.0137 & 0.0854\\
  \hline
\end{tabular}
\label{tableD}
\caption{The table above summarises the findings of our simulation tests on the Dirichlet problem. $\rho$ is the number of basis functions spanning the projection subspace, $c$ the number of samples used in the sketch, time in seconds is the time taken for 1000 sketched problem evaluations,  and ratio is the percentage of the rows of $X$ utilised in the sketch. The remaining quantities are relative values for the subspace approximation, sketching and overall solution errors and the condition number of $G$ averaged over 1000 FEM solutions. In all tests the parameter vectors were drawn from the uniform distribution $\mathcal{U}[10^{-1},10^2]$ apart from test D where $p$ was sampled from $\exp(-\mathcal{U}[10^{-4},1])$. Characteristic to these tests are the relative low overall error levels, due to the suppressed projection and subspace approximation errors in conjunction to the small condition number of the projected matrix.}
\end{table}

We first address the Dirichlet problem with a uniform boundary condition $u=0$ on $\partial \Omega$ yielding a FEM system with $n=29832$ degrees of freedom, one for each interior node in the mesh. The forcing term is taken to be a piecewise constant approximation of the function
$$
f(x_1,x_2,x_3) = \begin{cases} 
5 & \text{if} \; \sqrt{(x_1+\frac 1 2)^2 + x_2^2 + x_3^2} \leq 0.3,\\
0 & \text{otherwise},
\end{cases},
$$
in the interior of the domain. Two matrices $\Psi$ were constructed using the last $\rho=50$ and $\rho=100$ singular vectors of $\Delta$ resfor the needs of the tests referred to as A, B, C, D, E and F in table \ref{tableD}. In each of these tests we solve for the exact and the sketched solutions for 1000 parameter vectors and record their corresponding timings. The time for the sketched solution includes forming the sampling distribution, taking $c$ iid samples, sketching the matrix $G$ and solving the projected problem for $\hat r$. The particular settings for these tests and the average values of the errors obtained are tabulated in table \ref{tableD}. From this it appears that $\rho=50$ yields a sufficiently small projection error of about $6\%$ despite that $p$ varies over four orders of magnitude, while the overall relative error is bounded below $10\%$ in all tests. As anticipated, with the sampling budget increasing from $c=5000$ to $c=100000$ the sketching error $\|G - \hat G\|_F/\|G\|_F$ drops from 13\% to about 2\%, even though only 8\% of the rows of $X$ are sampled in the process, which indicates that the sampling is highly inhomogeneous. Finally, the times for 1000 sketched solutions were found to be in the range 500 - 600 s, yielding an average of about 0.55 s per FEM problem, which is substantially lower to the recorded average of 3.1 s for an exact high-dimensional solution. Critical to this desirable performance is the small condition number $\kappa(G) \approx 10$ which implies that $\lambda_{\min}(G)$ is bounded away from zero, according to the bound in the Theorem \ref{thm:totalerror3_second}. More insight into the dependence of the error components on the parameter vector can be obtained by the histograms in figure \ref{Dir_hist} illustrating the variation of the projection, sketching, subspace approximation and total errors, across the range of the simulated problems in test C, where $p$ was sampled from the uniform distribution $\mathcal{U}[10^{-1},10^2]$. At the same figure we plot also the histogram of the condition number of $G$ for $\rho=50$ and next to it that for $\rho=100$ for comparison, both of which indicate that $G$ is a well-conditioned matrix for all choices of the parameter vector.     
\begin{figure}
\begin{center}
\includegraphics[width=0.45\textwidth]{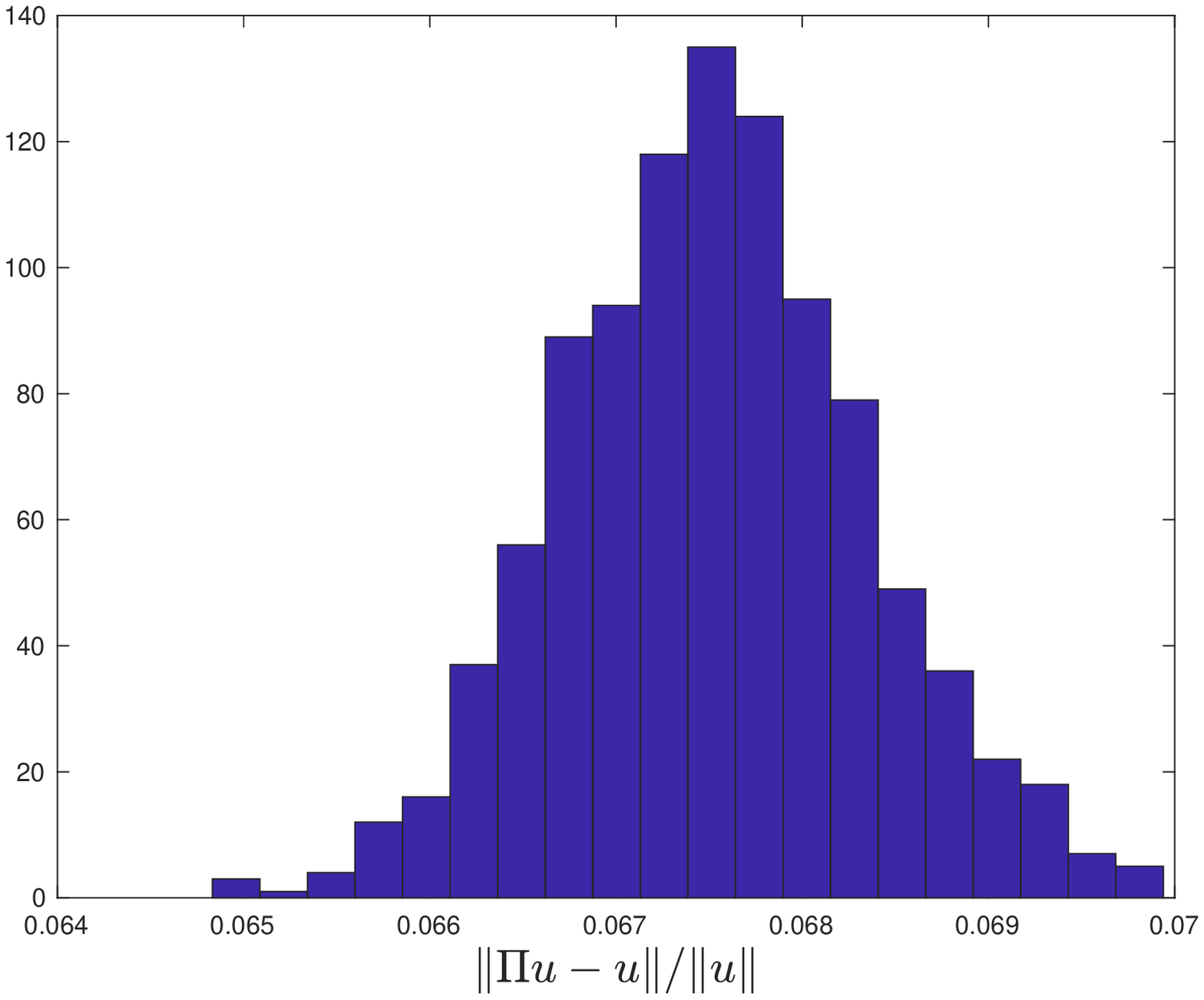} \includegraphics[width=0.45\textwidth]{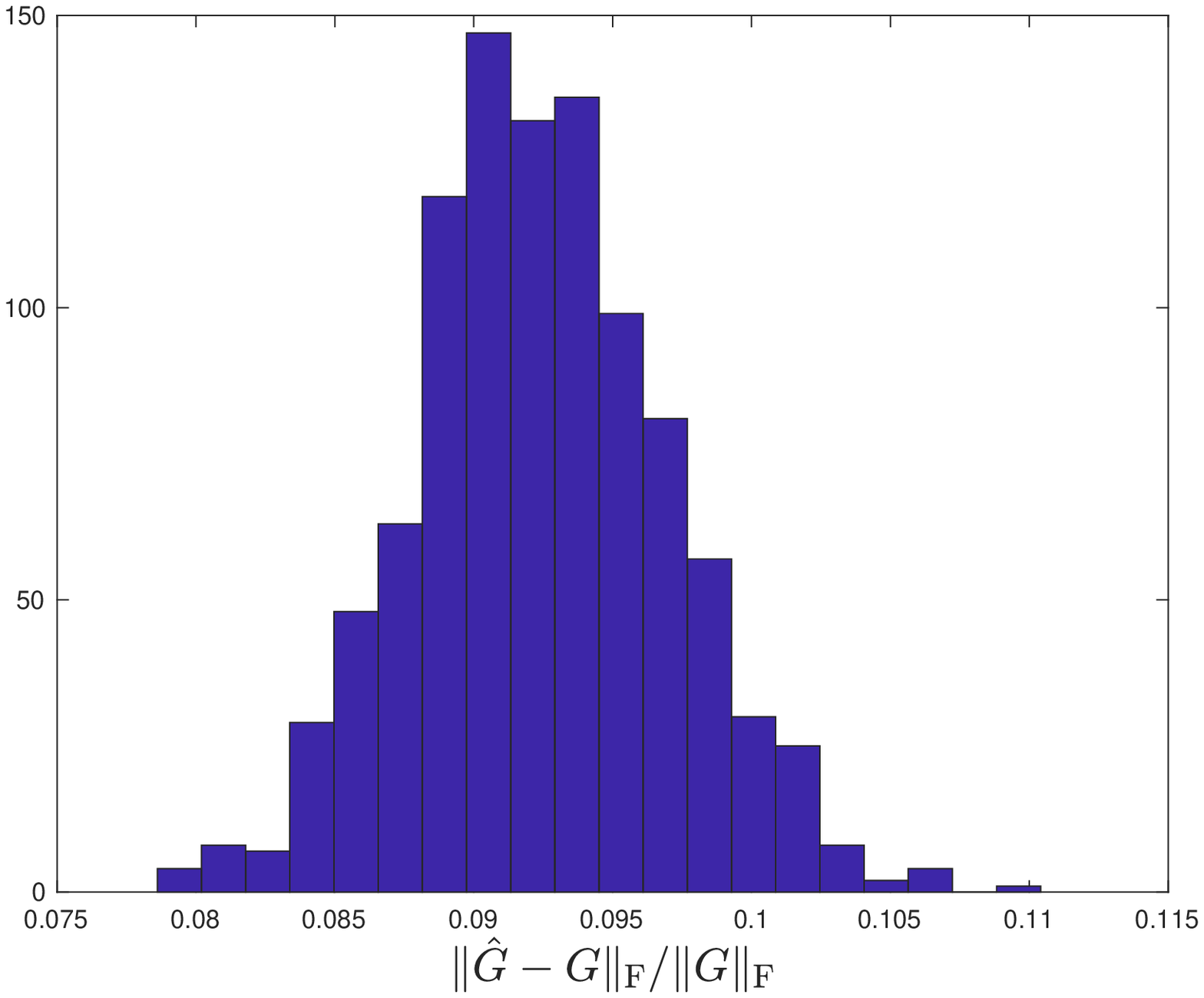}\\
\includegraphics[width=0.45\textwidth]{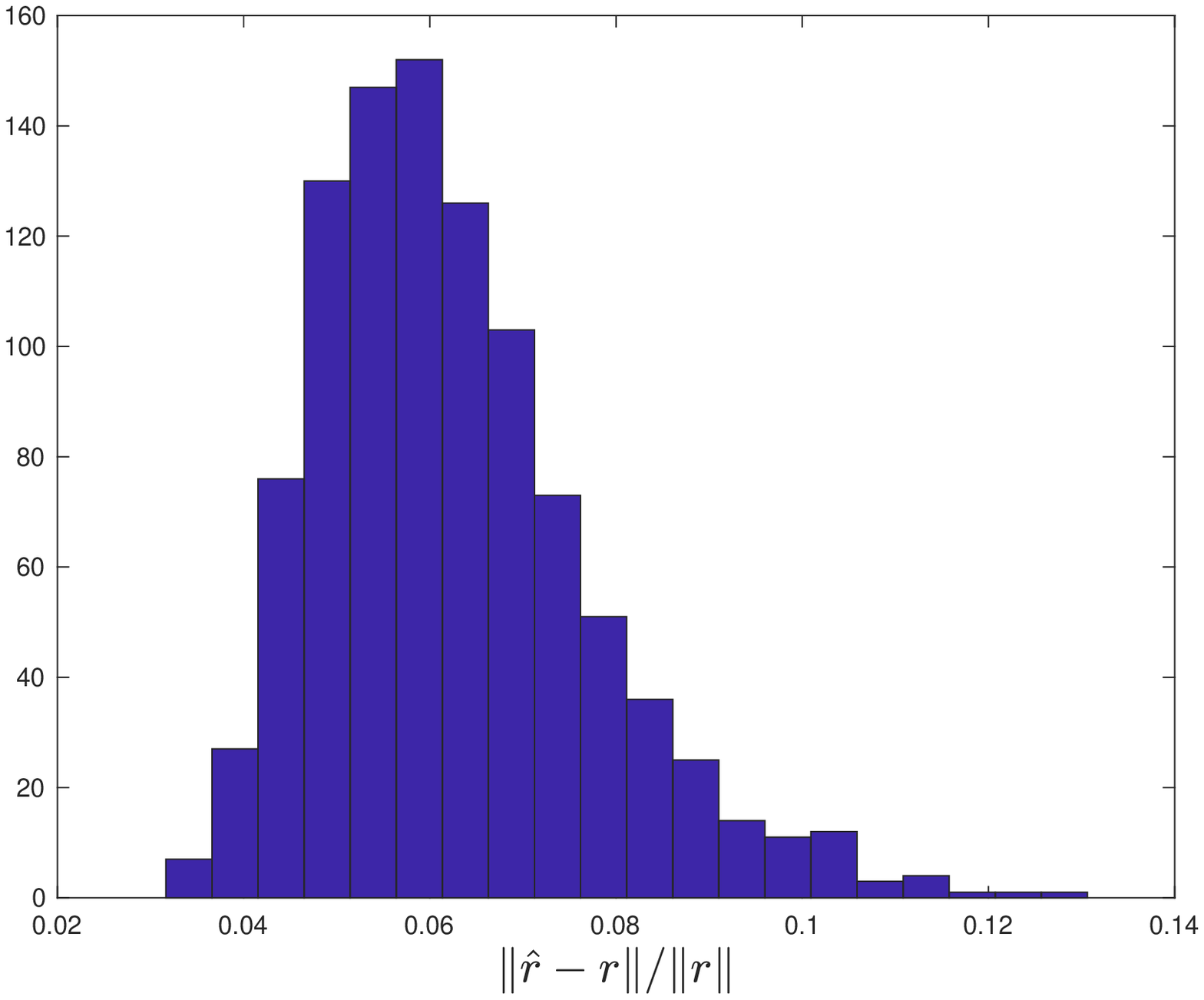} \includegraphics[width=0.45\textwidth]{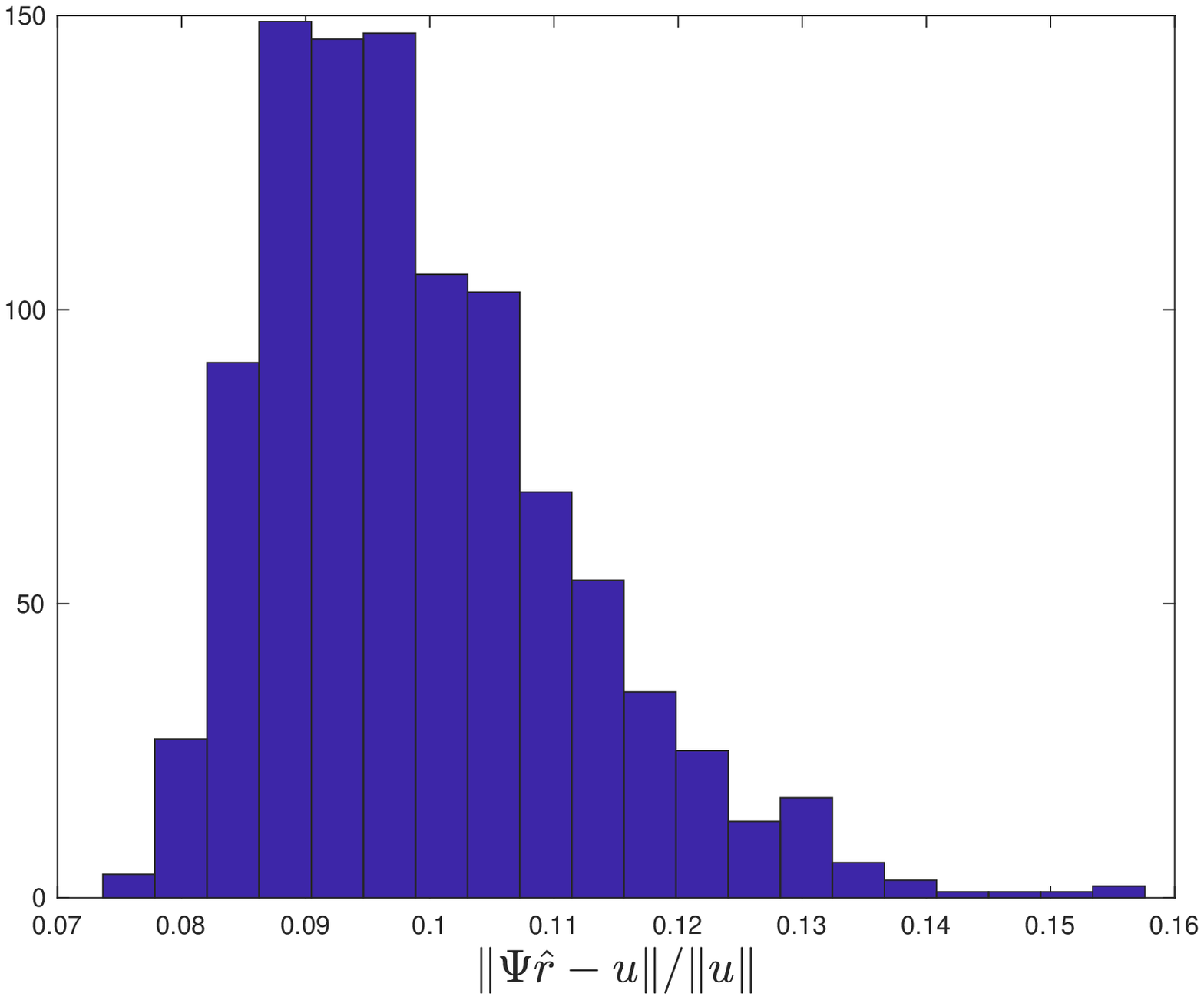}\\
\includegraphics[width=0.45\textwidth]{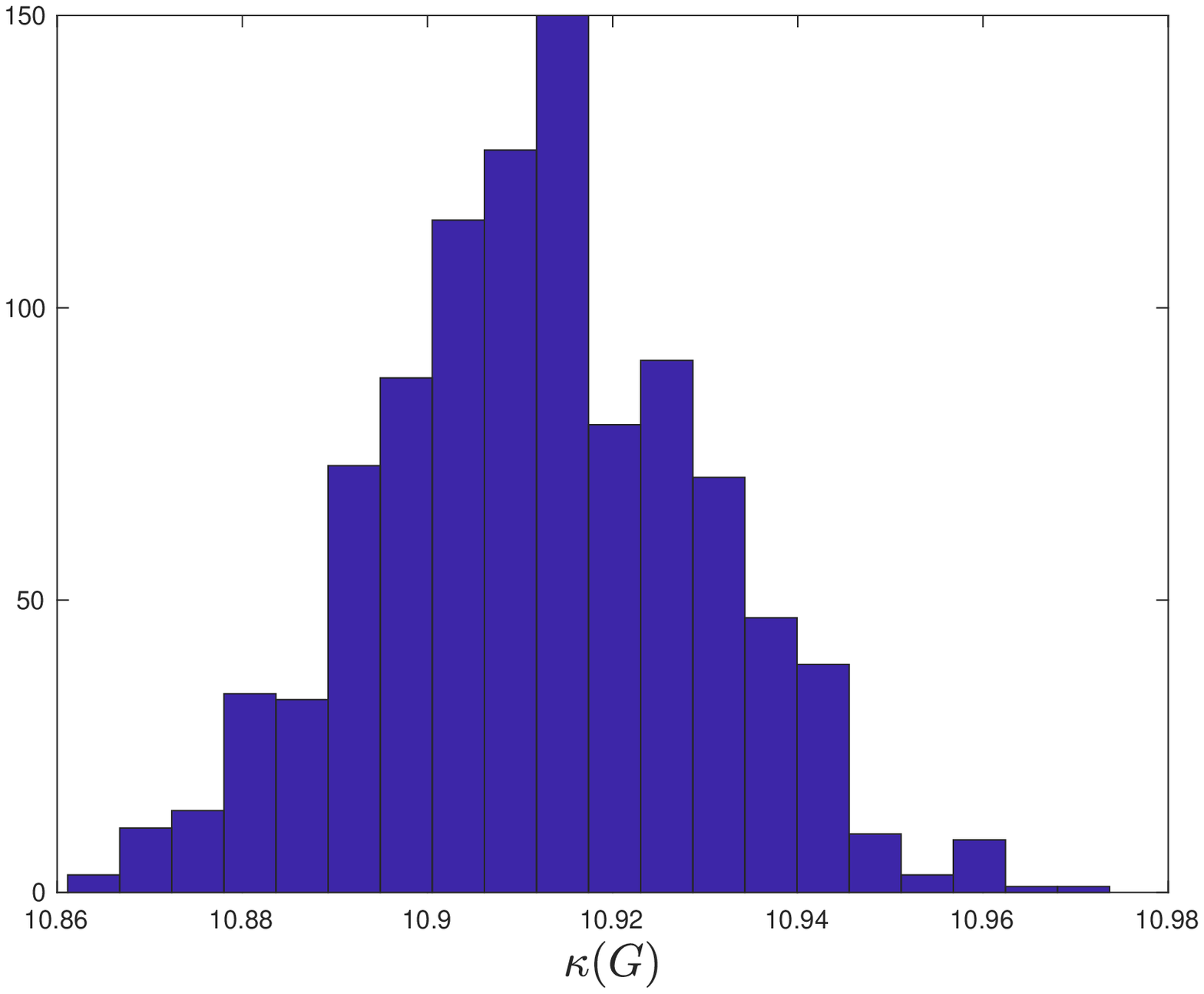}
\includegraphics[width=0.45\textwidth]{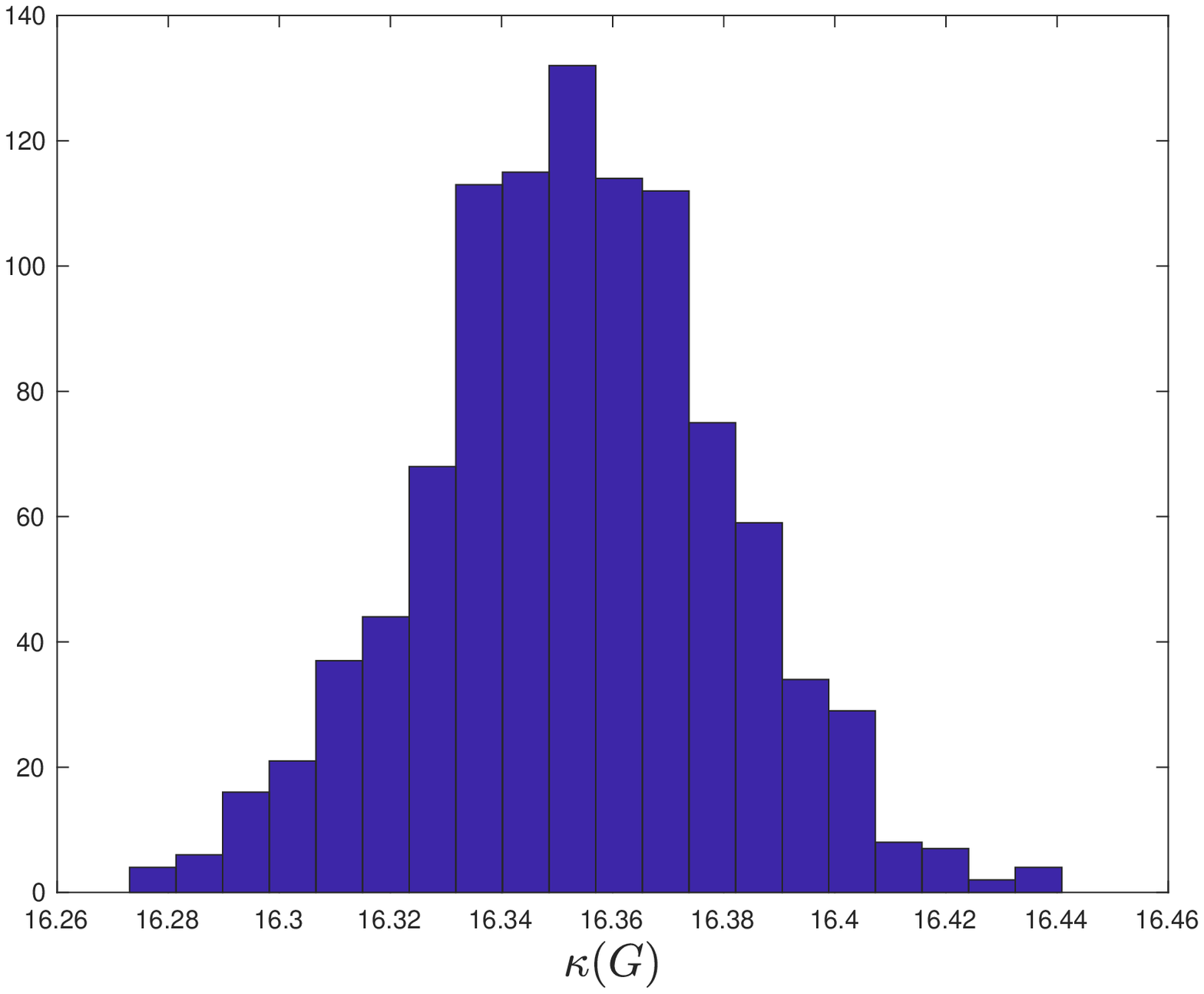}
\caption{Histograms depicting the relative variation of the projection, subspace approximation, simulation and total solution errors for the 1000 simulations in test C where $c=10000$, $\rho=50$ and $p$ was drawn from the uniform distribution $\mathcal{U}[10^{-1},10^2]$. The figures at the top row show that the projection and sketching errors are symmetrically concentrated around some small values without any outliers, while those in the second row for the subspace approximation and overall errors appear to be somewhat skewed towards zero. This desirable behaviour can be explained via the condition number of the projected matrix $G$ that controls the overall error amplification, as shown at the bottom left figure. For comparison, we plot to its right the respective histogram for $\rho=100$ indicating that $G$ remains a well-conditioned matrix for these choices of $\rho$.}
\label{Dir_hist}
\end{center}
\end{figure}                

\begin{figure}
\begin{center}
\includegraphics[width=0.47\textwidth]{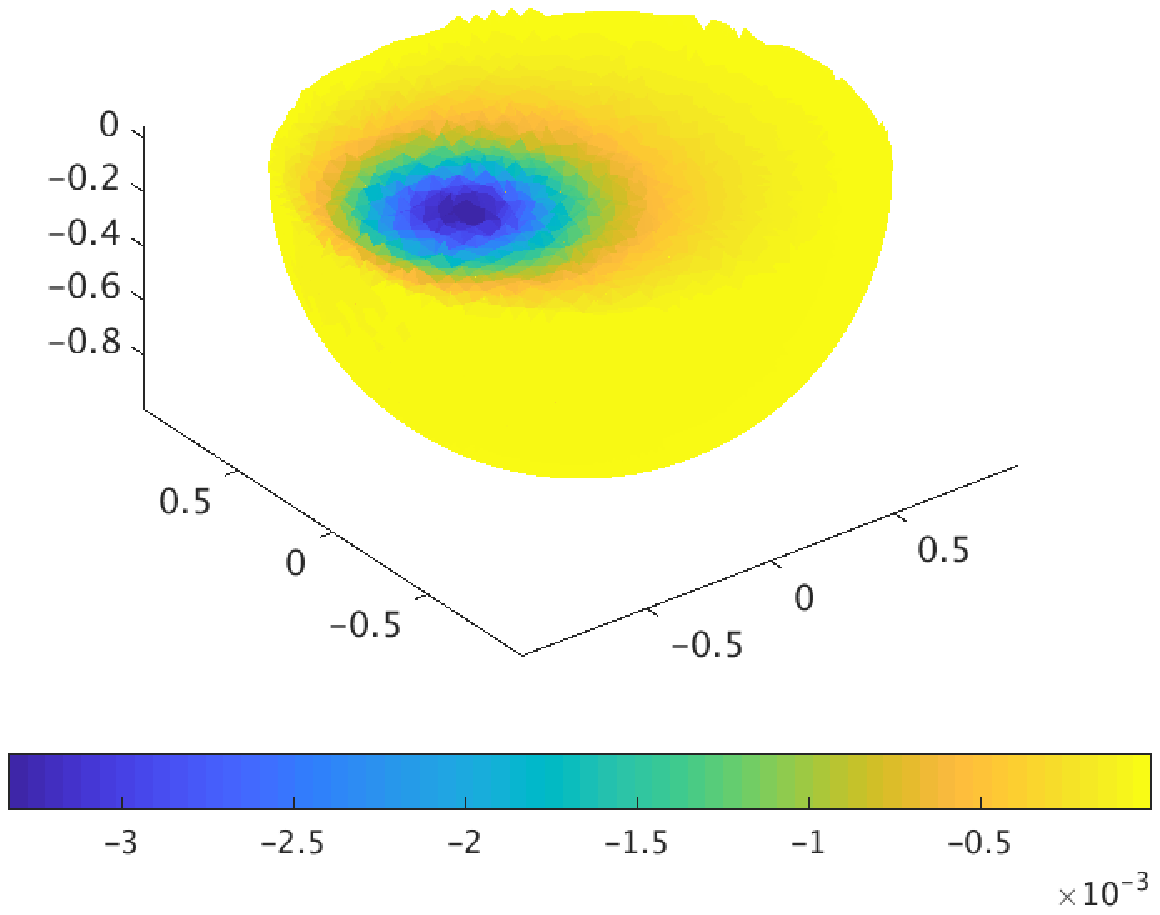} \includegraphics[width=0.47\textwidth]{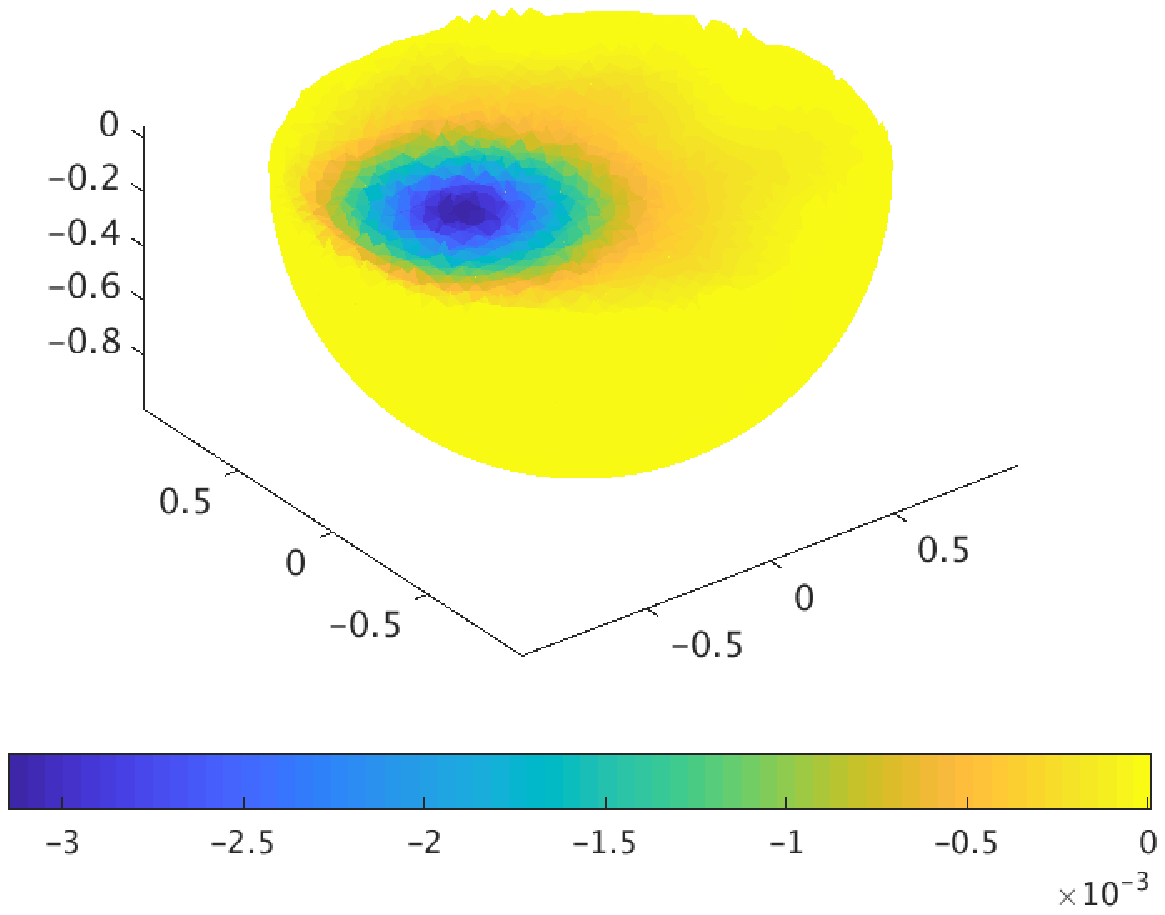}\\
\includegraphics[width=0.47\textwidth]{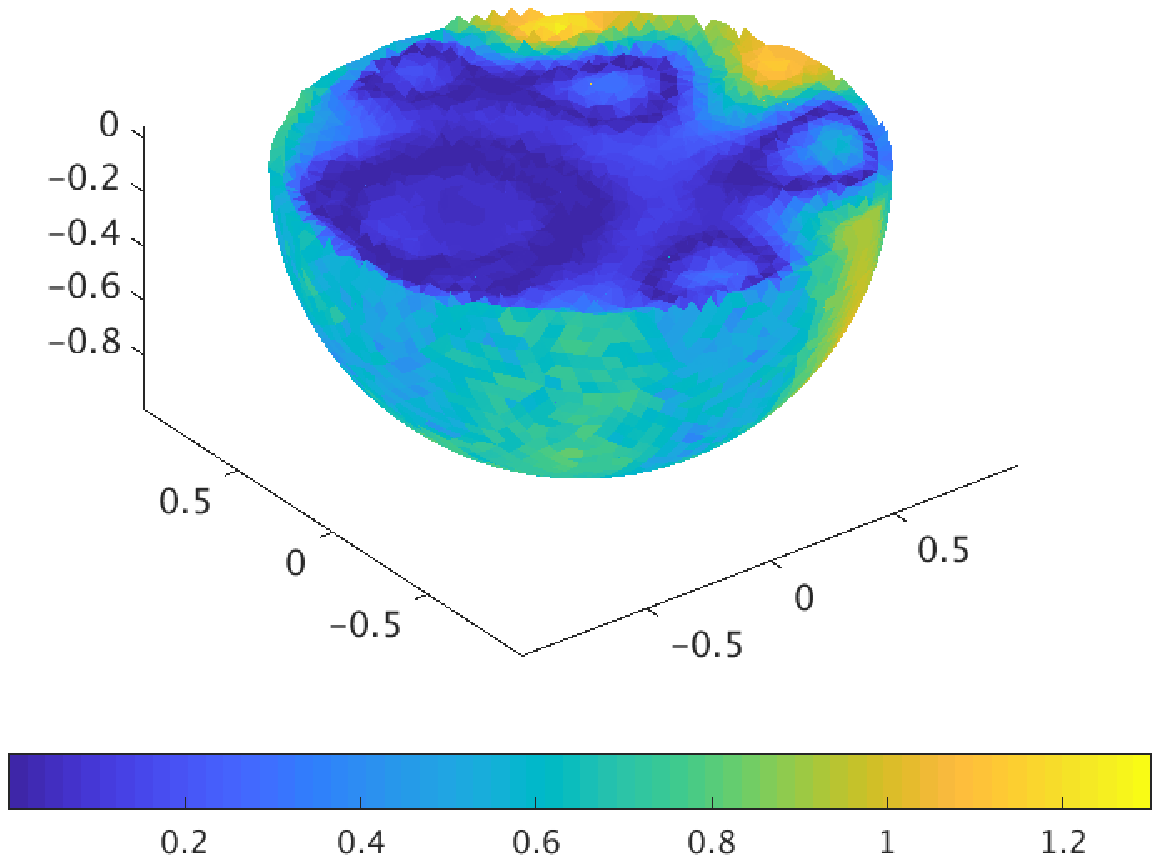}
\caption{At the top row, an extract of the exact $u^*$ (left) and sketched $\Psi \hat r$ (right) solutions of the Dirichlet problem at the bottom half of the domain for one instance of $p$ during test E. Below, the percentage relative error mapped to this section of the domain is shown below. The recorded times for the sketched and exact solutions were 0.4 s and 3.1 s respectively.}
\end{center}
\end{figure}

\subsection{The Neumann problem}

\begin{table}
\centering
\begin{tabular}{|c|c|c|c|c|c|c|c|c|c|}
  \hline
test & $\rho$ & $c$ & time & ratio & $\frac{\|\Pi u - u\|}{\|u\|}$ & $\frac{\|\hat G - G\|_F}{\|G\|_F}$ & $ \kappa(G)$ &  $\frac{\|\hat r - r\|}{\|r\|}$ & $\frac{\|\Psi \hat r - u\|}{\|u\|}$ \\ \hline \hline
A &  100  & 5000 & 600 & 0.0087 & 0.0040 & 0.2079 & 1728 & 0.4946 & 0.4418\\ 
B  & 100  & 50000 & 776 &   0.0814 & 0.0039 & 0.0649 & 1743 &  0.1107 & 0.1365\\
C &  50  & 100000 & 605 &  0.1539 & 0.0053 & 0.0293 & 1153 & 0.0873 & 0.1294\\
D &  50  & 100000 & 562 & 0.1574 & 0.0053 &  0.0293  & 1062 &  0.0792 & 0.1204\\
E & 50 & 500000 & 1897 & 0.5496 & 0.0053 & 0.0131 & 1130 & 0.0375 & 0.1126\\
F & 50 & 500000 & 1133 & 0.5085 & 0.0053 & 0.0131 & 1055 & 0.0383 & 0.1223\\\hline
\end{tabular}
\label{tableN}
\caption{The table above summarises the findings of our simulation tests on the Neumann problem. $\rho$ is the number of basis functions spanning the projection subspace, $c$ is the number of samples used in the sketching, time is the duration in seconds taken for a 1000 sketched problem evaluations and ratio is the percentage of the rows of $X$ utilised in the sketch. The remaining quantities are relative errors for the projection, simulation, subspace approximation, overall solution error, and the condition of the projected matrix $G$ averaged over 1000 problem solutions. The parameter vectors were drawn from the uniform distribution $\mathcal{U}[10^{-1},10^2]$, apart from tests D and E where $\exp(-\mathcal{U}[10^{-4},1])$ was invoked.}
\end{table}

For the Neumann problem we consider a forcing term $f=0$ in the interior of the domain and the condition 
$$
g^{(N)}(x_1,x_2,x_3) = \begin{cases} 
1 & \text{if} \; \sqrt{x_1^2 + (x_2-1)^2 + x_3^2} \leq 0.4\\
0 & \text{otherwise}
\end{cases}, 
$$
at the boundary. Similarly to the Dirichlet case we set to investigate the performance of our algorithm in approximating $u^*$ on a series of tests whose results are tabulated in table \ref{tableN}. To aid the comparison with the Dirichlet results the same mesh is used, however at the Neumann problem $u^*$ has $n+n_\partial-1 = 34048$ degrees of freedom, incorporating all nodes of the mesh apart from one whose value is fixed in order to enforce uniqueness. Overall, the error values recorded show that despite the very small projection error, the total errors observed are substantially larger to those at the Dirichlet tests. Increasing the sampling budget to 500000, sampling  54\% of the rows of $X$, suppresses the total error to around 12\% with a linear reduction in the sketching error. However the relative regression solution error, and thus the total solution error remain large and noticeably, the average condition number of $G$ is about two orders of magnitude larger to that in the Dirichlet tests. The large values of $\kappa(G)$, manifesting that $\lambda_{\rho}(G)$ shrinks substantially, are also confirmed in the associated histogram plots for test F in figure \ref{Neu_hist}. In terms of the computational times, the sketched approach maintains its advantage against the deterministic solution since for $c=100000$ we approximate a solution with about 12\% error in about 0.6 s while the corresponding $u^*$ takes 4.2 s.


                      
\begin{figure}
\begin{center}
\includegraphics[width=0.45\textwidth]{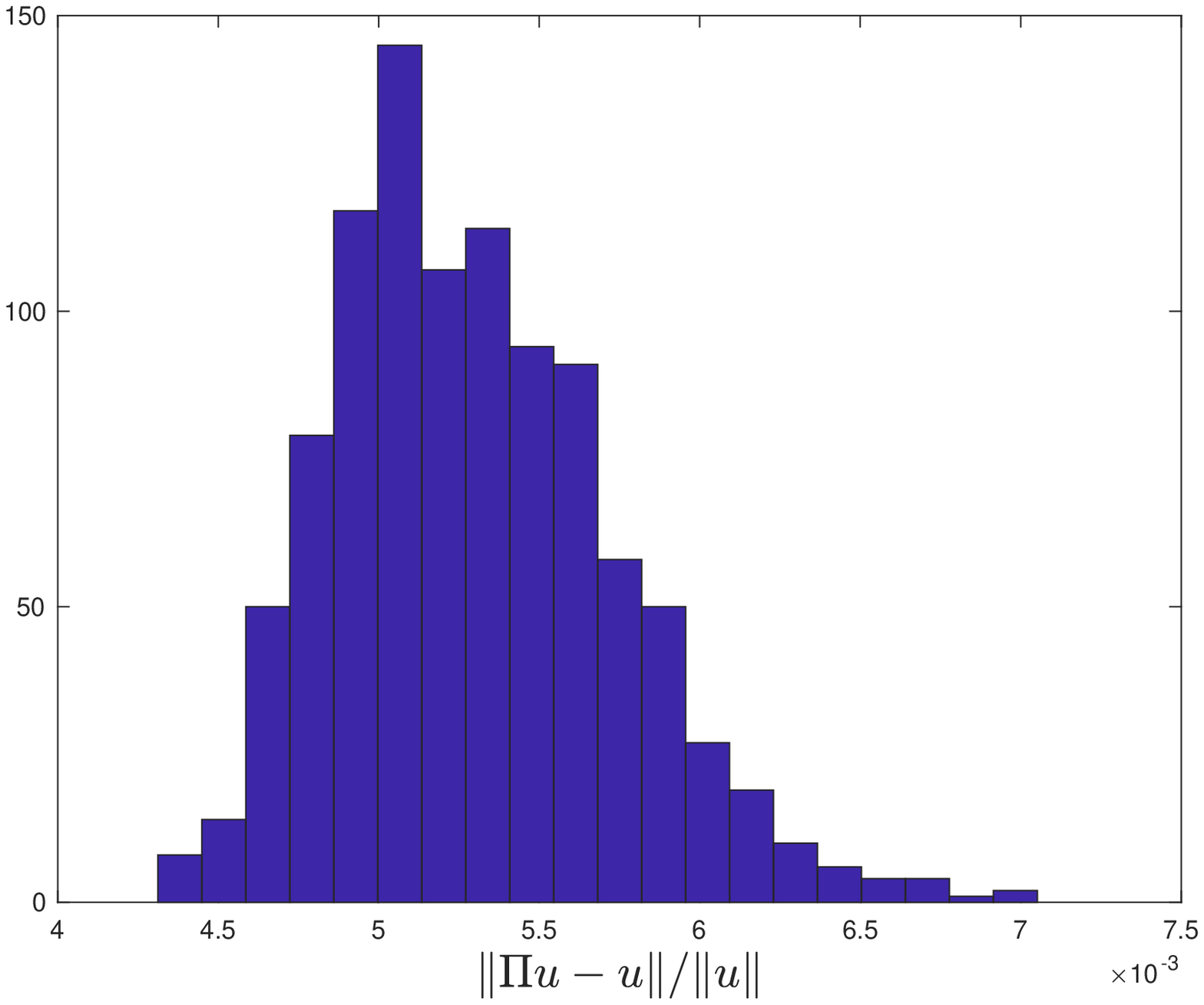} \includegraphics[width=0.45\textwidth]{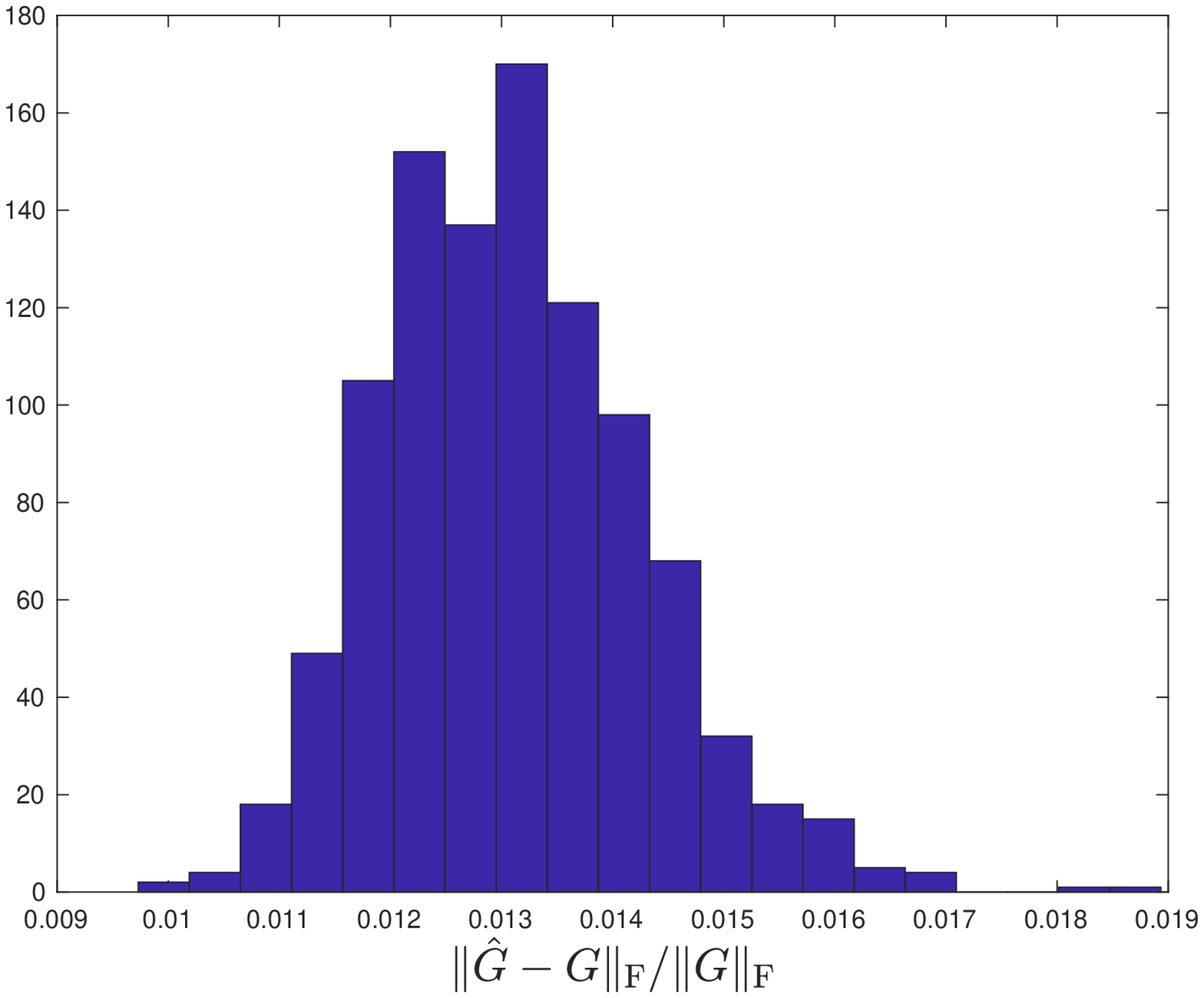}\\
\includegraphics[width=0.45\textwidth]{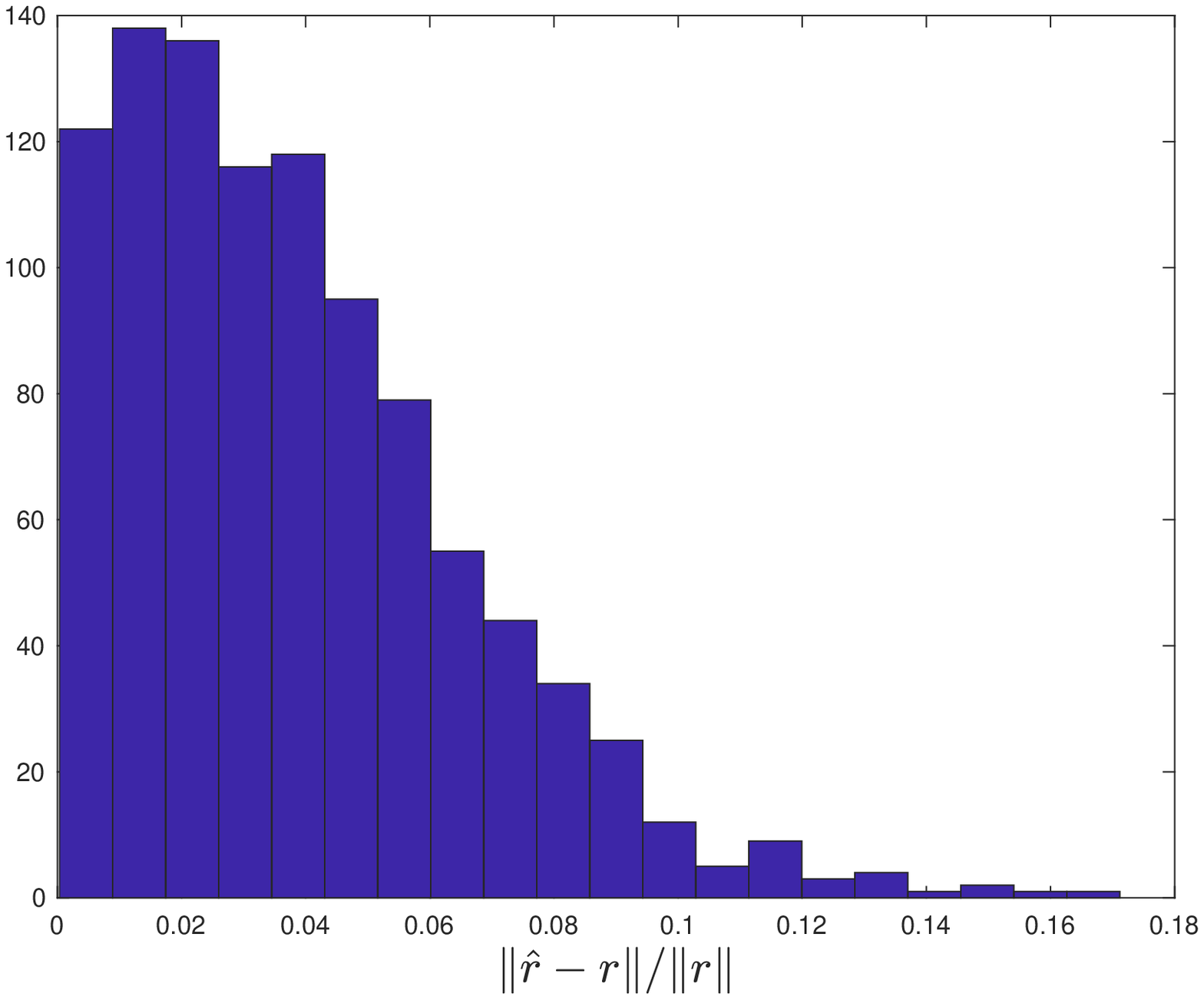} \includegraphics[width=0.45\textwidth]{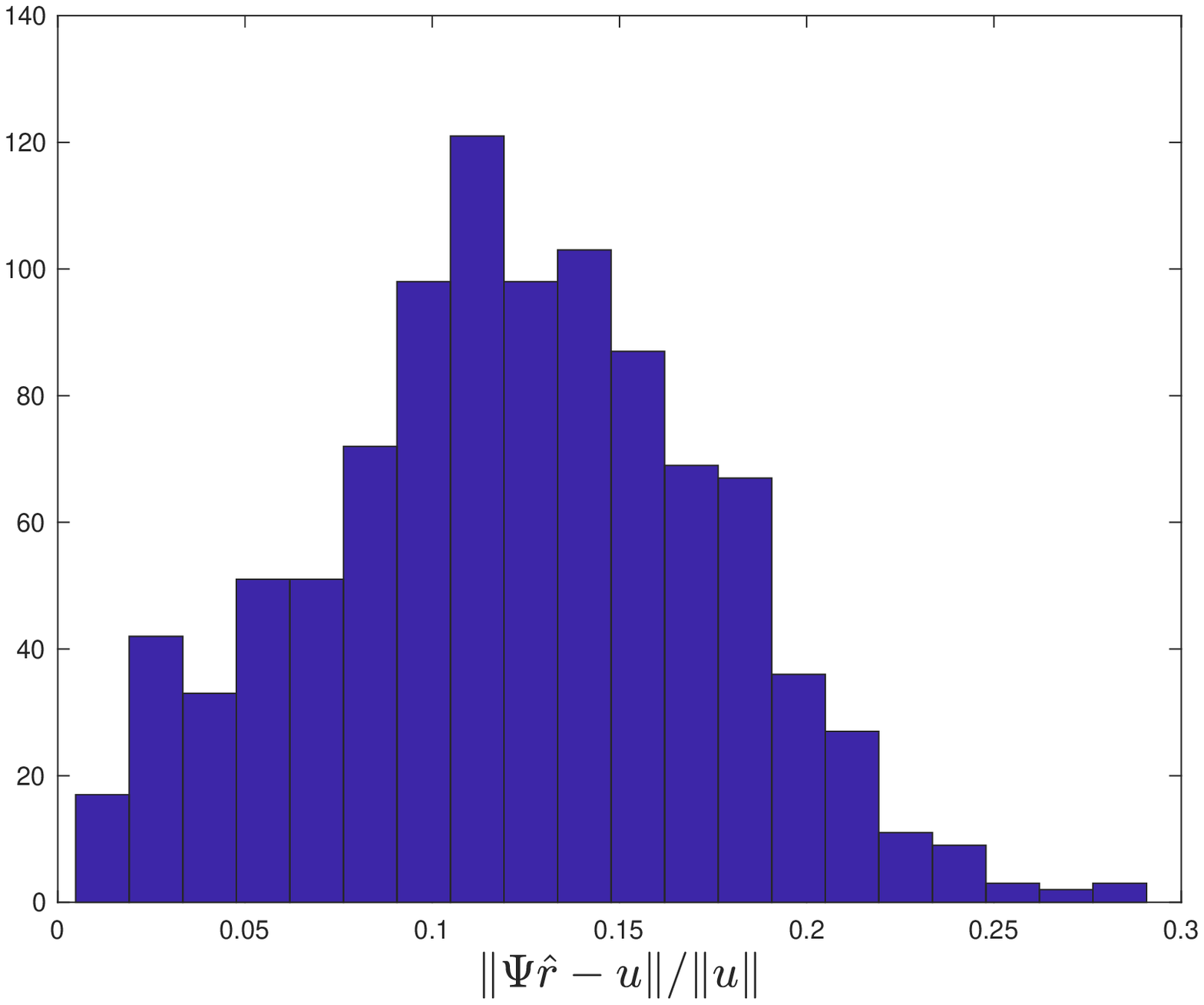}\\
\includegraphics[width=0.45\textwidth]{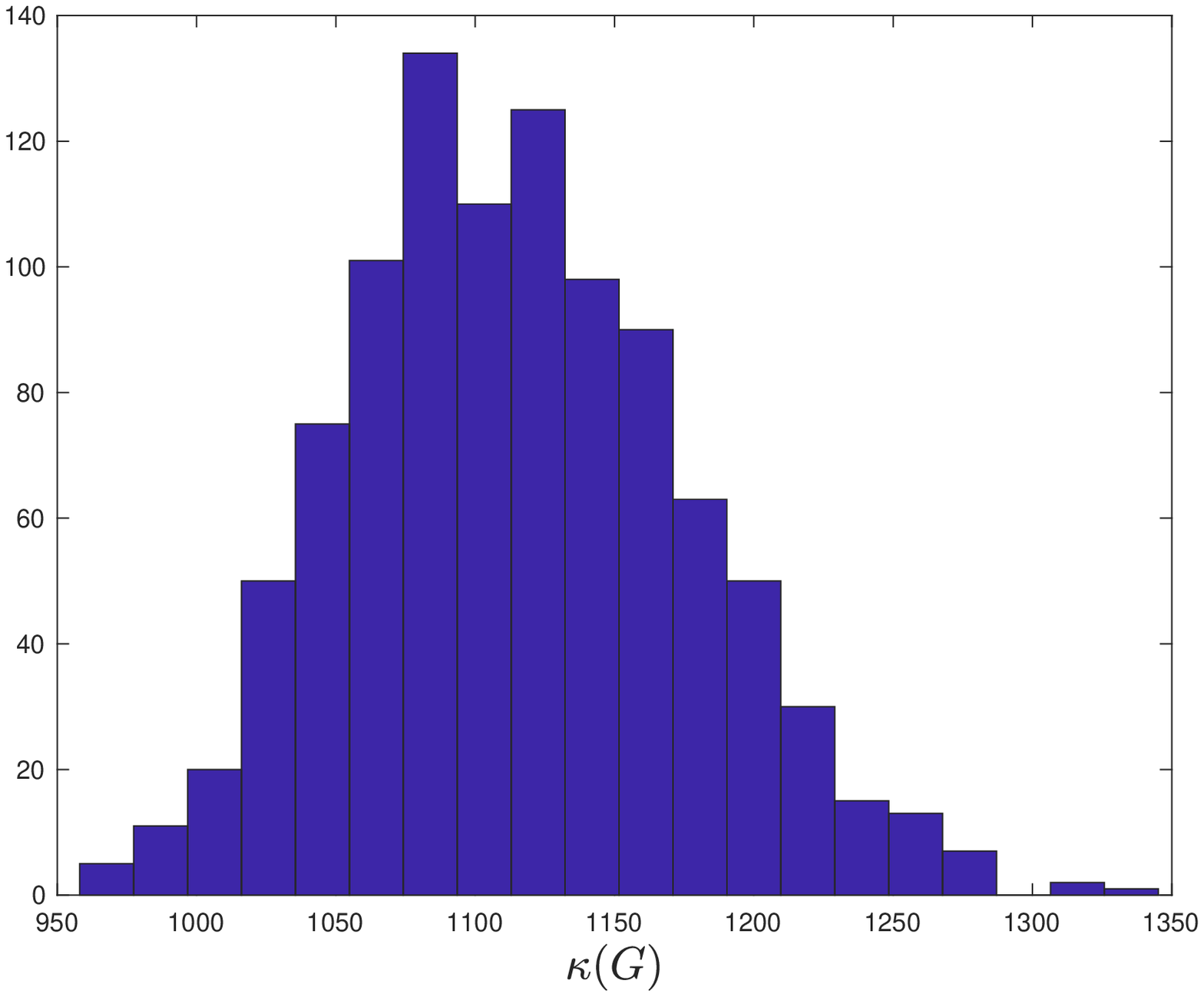}
\includegraphics[width=0.45\textwidth]{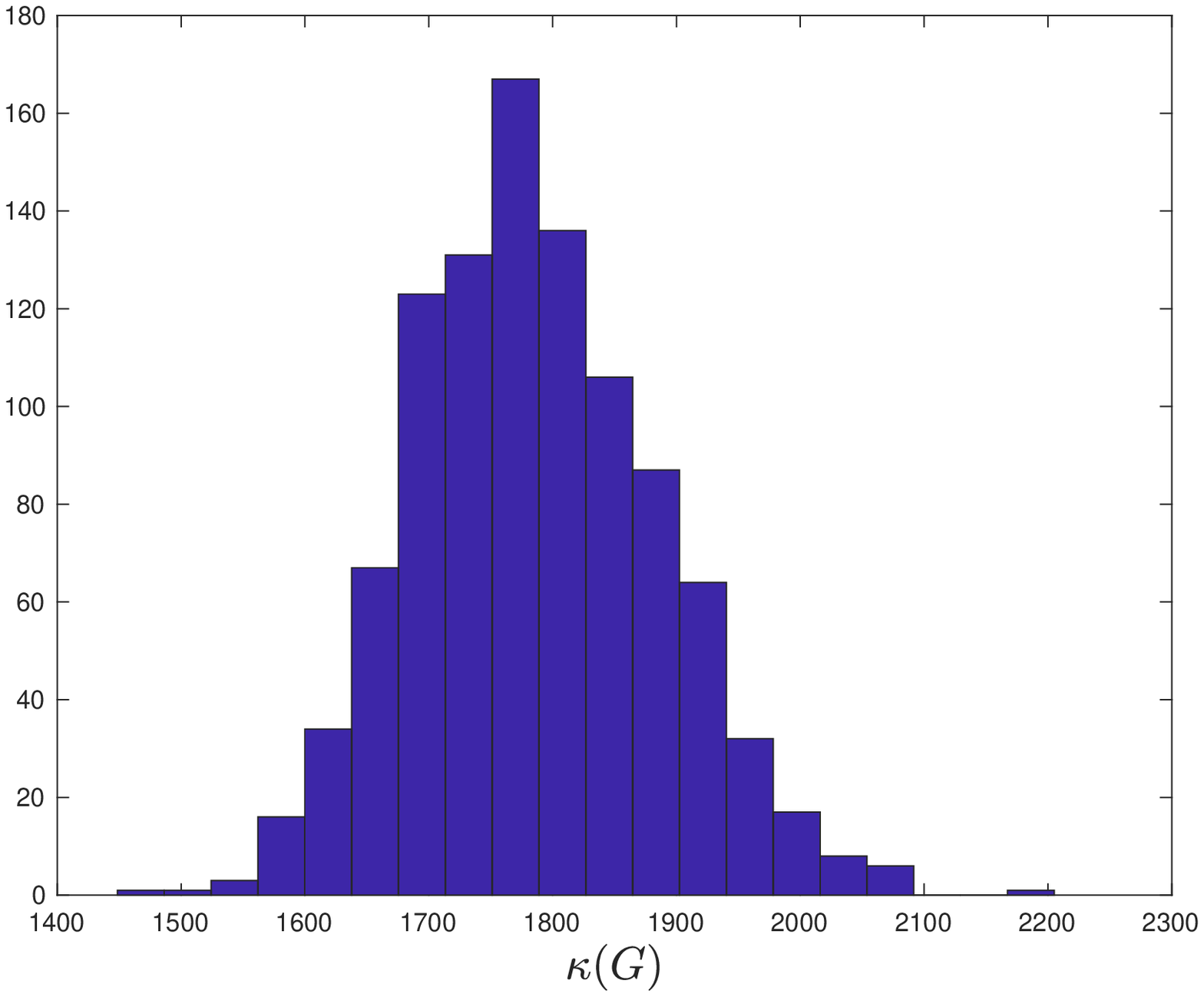}
\caption{Histograms of the various error components affecting the sketched solution of the Neumann problem in 1000 simulations of test F, where $p \sim \mathcal{U}[10^{-1},10^2]$, $c=500000$ and $\rho=50$. Notice that the total error spans over a larger range of values as a affected by the relatively large condition number of the $G$ matrices, as shown at the bottom left figure. For comparison we plot to its right the respective histogram for $\rho=100$ showing that $\kappa(G)$ remains significantly higher compared to the Dirichlet case and that the situation worsens as $\rho$ increases.}
\end{center}
\label{Neu_hist}
\end{figure}    

\section{Conclusions}

We propose a fast, randomised implementation of the finite element method for solving elliptic partial differential equations on high-dimensional models. Our approach is particularly appealing to the many query context where the solution to the PDE is sought for many instances of the parameter vector. We reformulate the linear FEM system as an overdetermined least squares problem and then apply an orthogonal projection onto a low-dimensional subspace. Invoking this projection offers a twofold advantage: it reduces the dimensionality of the problem but it also allows to sketch the matrices involved using sampling distributions that approximate well those corresponding to the statistical leverage scores. We  analyse the conditions on the subspace that enable this favourable performance and then we bound the errors imparted to the solution. This led to the conclusion that the error amplification is controlled by the condition number of the coefficients matrix of the projected problem. Tested on the Dirichlet and Neumann problems for the steady-state diffusion equation boundary value problem, the performance of the algorithm is aligned to the derived approximation bounds, while it yields substantial computational savings for a moderate solution error.

\section*{Acknowledgements}
NP and YW are grateful to EPSRC for funding this work through the project EP/R041431/1, titled `Randomness: a resource for real-time analytics'. NP acknowledges additional support from the Alan Turing Institute.

\section*{Appendix}

\begin{proof}[Proof of Proposition \ref{prop:unbias}] The claim can be shown through the linearity property of expectation and its definition as follows
\begin{align*}
\mathbb{E}_{\xi}[\hat G]&=\mathbb{E}_{\xi}\Big [\frac 1 c \sum_{t=1}^c \frac{1}{\xi_{r_t}} z_{r_t} (D \Psi)^T_{(r_t)*}(D \Psi)_{(r_t)*}\Bigr ]\\
&= \sum_{t=1}^{c} \sum_{\ell=1}^{kd} \xi_{\ell}\frac{z_\ell}{c\xi_{\ell}} (D\Psi)^T_{(\ell)*}(D\Psi)_{(\ell)*}\\
&=\sum_{\ell=1}^{kd} z_{\ell}(D\Psi)^T_{(\ell)*}(D\Psi)_{(\ell)*}=G.
\end{align*}
\end{proof}

\begin{proof}[Proof of Corollary \ref{col:unbias}] Recall that $SS^T$ is a $kd\times kd$ diagonal matrix, whose $\ell$th index has a value of the product of the cardinality of index $\ell$ in the sample set $\{r_1,\ldots, r_t\}$ denoted by $c_\ell$ with $\frac{1}{c\xi_\ell}$. From this it remains to show that each entry has expectation $1$.  For an index $\ell$, the sampling procedure can be treated as a sequence of $c$ binomial trials each with success probability $\xi_\ell$. Therefore we have 
$$\mathbb{E}_\xi[(SS^T)_{\ell \ell}]=\mathbb{E}\big[\frac{c_\ell}{c\xi_\ell}\big]=\frac{1}{c\xi_\ell}\mathbb{E}[c_\ell]=\frac{1}{c\xi_\ell}c\xi_\ell = 1,
$$
where the penultimate equality is by virtue of the properties of binomial random variables.
\end{proof}

\begin{proof}[Proof of Proposition \ref{prop:minsigma}] 
First we note that for all $1\leq i,j\leq \rho$ we have
$$
\hat G_{ij}= \sum_{t=1}^c (G_t)_{ij},
$$ 
where
$$
(G_t)_{ij} = \frac{1}{c\xi_{r_t}} z_{r_t}(D \Psi)_{r_t i}(D \Psi)_{r_t j}
$$
A direct consequence of $\mathbb{E}_{\xi}[\hat G]=G$ is that $\mathbb{E}_{\xi}[\hat G_{ij}]=G_{ij}$, and $\mathbb{E}_\xi [(G_t)_{ij}] = \frac 1 c G_{ij}$. We then have that
$$
\mathrm{Var}_\xi [\hat G_{ij}] = \sum_{t=1}^c \mathrm{Var}_\xi [(G_t)_{ij}] = \sum_{t=1}^c \Bigl ( \mathbb{E}[(G_t)_{ij}^2 ] - \mathbb{E}[(G_t)_{ij}]^2 \Bigr )
$$
from where we get
$$
\mathrm{Var}_\xi[\hat G_{ij}] = \frac 1 c \Bigl ( \sum_{\ell=1}^{kd} \frac{1}{\xi_\ell} z_\ell^2 (D\Psi)^2_{\ell i} (D\Psi)^2_{\ell j} - G_{ij}^2 \Bigr ) \quad 1 \leq i,j \leq \rho.
$$
To bound the sketching-induced error we have
$$
\mathbb{E}_\xi[\|G-\hat G\|_F^2]=\sum_{i=1}^\rho \sum_{j=1}^\rho \mathbb{E}_\xi[(G-\hat G)^2_{ij}]=\sum_{i=1}^\rho \sum_{j=1}^\rho \mathrm{Var}_{\xi}[\hat G_{ij}].
$$
In fixing the matrix indices we have
\begin{align}\nonumber
\sum_{i=1}^\rho \sum_{j=1}^\rho\mathrm{Var}_{\xi}[\hat G_{ij}] & = \sum_{i=1}^\rho \sum_{j=1}^\rho \frac{1}{c} \Bigl ( \sum_{\ell=1}^{kd}\frac{z_\ell^2}{\xi_\ell} (D\Psi)^2_{\ell i}(D\Psi)^2_{\ell j} - G_{ij}^2 \Bigr)\\\nonumber
& = \frac 1 c \sum_{\ell=1}^{kd} \frac{z_\ell^2}{\xi_\ell} \sum_{i=1}^\rho \sum_{j=1}^\rho (D \Psi)^2_{\ell i} (D \Psi)^2_{\ell j} - \frac 1 c \sum_{i=1}^\rho \sum_{j=1}^\rho G_{ij}^2 \\\nonumber
& = \frac 1 c \sum_{\ell=1}^{kd} \frac{z_\ell^2}{\xi_\ell} \sum_{i=1}^\rho (D\Psi)^2_{\ell i} \sum_{j=1}^\rho (D\Psi)^2_{\ell j} - \frac 1 c \|G\|^2_F\\\label{optxi}
& = \frac 1 c \Bigl ( \sum_{\ell=1}^{kd} \frac{z_\ell^2}{\xi_\ell} \|(D\Psi)_{(\ell)*}\|^4 - \|G\|^2_F \Bigr )\\ 
& \leq \frac 1 c \Bigl ( \sum_{\ell=1}^{kd} \frac{z_\ell^2}{\xi_\ell} \|D_{(\ell)*}\|^4- \|G\|^2_F \Bigr ) \label{suboptxi}
\end{align}
where the first inequality holds true since $\mathrm{Var}[X]\leq \mathbb{E}[X^2]$ for any real-valued random variable $X$, and the last from $\Psi^T\Psi=I$. In order to optimise the choice of $\xi$ in reducing the sketching error we invoke the Lagrangian function based on \eqref{optxi}
\begin{align*}
\mathcal{L}(\xi;\lambda)= \sum_{\ell=1}^{kd} \frac{z_\ell^2}{\xi_\ell} \|(D\Psi)_{(\ell)*}\|^4 + \lambda\Big(\sum_{\ell=1}^{kd} \xi_{\ell} -1\Big).
\end{align*}
for which the method of Lagrange multipliers returns 
 \begin{align*}
\xi_\ell = \frac{z_\ell \|(D\Psi)_{(\ell)*}\|^2}{\sum_{\ell=1}^{kd}z_\ell \|(D\Psi)_{(\ell)*}\|^2},  \quad 1 \leq \ell \leq kd.
\end{align*}
Plugging in the optimal expression of $\xi$ into simulation error expression yields
\begin{align*} 
&\mathbb{E}_\xi[\|G-{\hat G}\|^2_F]\leq  \frac{1}{c}\sum_{\ell=1}^{kd} z_{\ell} \|(D\Psi)_{(\ell)*}\|^2 \sum_{\ell=1}^{kd}z_\ell \|(D \Psi)_{(\ell)*}\|^2 - \frac 1 c \|G\|^2_F \\
&\leq \frac{1}{c} \Bigl (\bigl(\sum_{\ell=1}^{kd} z_\ell \|(D\Psi)_{(\ell)*}\|\bigr)^2 - \|G\|^2_F\Bigr ). \end{align*}
\end{proof}

\begin{proof}[Proof of Proposition \ref{prop:max}]
Applying the induced norm to the expression of $\hat G$ in \eqref{hatA} yields 
\begin{align*}
\|\hat G\|& \leq \frac 1 c \sum_{t=1}^c \frac{z_{r_t}}{\xi_{r_t}} \bigl \|(D\Psi)^T_{(r_t)*}(D\Psi)_{(r_t)*} \bigr \|=\frac 1 c \sum_{t=1}^c \frac{z_{r_t}}{\xi_{r_t}}  \bigl \|(D\Psi)_{(r_t)*} \bigr\|^2\\
& = \frac 1 c \sum_{t=1}^c \sum_{\ell=1}^{kd} z_\ell \bigl \|(D\Psi)_{(r_\ell)*} \bigr\|^2  =\sum_{\ell=1}^{kd} z_\ell \bigl \|(D\Psi)_{(r_\ell)*} \bigr\|^2 \leq d \, p_\Omega \|D\|^2. 
\end{align*}
where $p_\Omega =  \sum_{\ell=1}^k p_\ell |\Omega_\ell|$.
\end{proof}

\begin{proof}[Proof of Proposition \ref{prop:min}]
Applying directly the eigenvalue perturbation result from \cite{Meyer00}, we immediately have 
\begin{align}\label{eqn:perturb}
 \sigma_i (\hat G )\geq  \sigma_i ( G)+\lambda_{\min}(\hat G- G), \quad \text{for} \quad i=1,\ldots, \rho
\end{align}
where $\lambda_{\min}$ represents the minimum eigenvalue of a matrix. Note for the symmetric matrix $\hat G- G$, $|\lambda_{\min}(\hat G- G)|\leq \|\hat G- G\|$. Markov's inequality leads to
\begin{align*}
\mathbb{P}_{\xi}\big( \|\hat G- G\|\leq  \gamma\sigma_{\min} ( G )\big)\geq 1-\min\Big\{1,\frac{\mathbb{E}_{\xi}[\|\hat G- G\|_F]}{\gamma\sigma_{\min}(G )}\Big\}.
\end{align*}
Thus with the above indicated probability, we have $$|\lambda_{\min}(\hat G-G)|\leq \gamma\sigma_{\min} ( G),$$
which implies that $\lambda_{\min}(\hat G-G)\geq -\gamma\sigma_{\min} ( G )$. Substituting back into \eqref{eqn:perturb} yields the final assertion.
\end{proof}

\begin{proof}[Proof of Lemma \ref{lem:inhom}]
With the definitions as before, i.e, $Y=Z^{\frac 1 2}D$, $X=Y \Psi$ and $\Psi^T\Psi = I$, let $
X = U_X \Sigma_X V_X^T$, where $U_X \in \mathbb{R}^{kd \times \rho}$, $\Sigma_X \in \mathbb{R}^{\rho \times \rho}$, $V_X \in \mathbb{R}^{\rho \times \rho}$, and $\beta = \|X\|_F^{-2}$. Then
$$
\|\xi^{l(X)} - \xi^{r(X)}\| = \Bigl \|\frac 1 \rho l_X - \beta r_X \Bigr \| = \Bigl \| \mathrm{diag}\Bigl (U_X \Bigl (\frac 1 \rho I - \beta \Sigma_X^2 \Bigr ) U_X^T \Bigr) \Bigr \|,
$$
where $\|\cdot\|$ can now be taken as an arbitrary vector norm to be determined. Taking the 2-norm gives
\begin{align*}
  &  \|\xi^{l(X)} - \xi^{r(X)}\|=\Bigl \| \mathrm{diag}\Bigl (U_X \Bigl (\frac 1 \rho I - \beta \Sigma_X^2 \Bigr ) U_X^T \Bigr) \Bigr \|_F\\
  &\leq\Bigl \| U_X \Bigl (\frac 1 \rho I - \beta \Sigma_X^2 \Bigr ) U_X^T \Bigr \|_F\leq \|U_X\|^2\Bigl \| \frac 1 \rho I  - \frac{  \Sigma_X^2}{\|X\|_F^2} \Bigr \|_F\\
  &\leq \sqrt{\sum_{i=1}^\rho \big(\frac 1\rho- \pi_i(\rho)\big)^2}=\sqrt{\sum_{i=1}^\rho \pi_i(\rho)^2-\frac{1}{\rho}}.
\end{align*}
On the other hand, taking the max-norm yields 
\begin{align*}
  &  \|\xi^{l(X)} - \xi^{r(X)}\|_{\max}=\Big\| \mathrm{diag}\Bigl (U_X \Bigl (\frac 1 \rho I - \beta \Sigma_X^2 \Bigr ) U_X^T \Bigr) \Big\|_{\max} \\
  &\leq \Big\|U_X \Bigl (\frac 1 \rho I - \beta \Sigma_X^2 \Bigr ) U_X^T  \Big\| \leq \|U_X\|^2\Bigl \| \frac 1 \rho I  - \frac{  \Sigma_X^2}{\|X\|_F^2} \Bigr \| \\
  &\leq \big(\max_i \pi_i(\rho)-\frac{1}{\rho}\big)\vee \big(\frac{1}{\rho}-\min_i \pi_i(\rho)\big).
\end{align*}
\end{proof}

\begin{proof}[Proof of Theorem \ref{thm:inhomo2}]
We have that
\begin{align}\label{eqn:Xboundin1}
\begin{split}
  & \|l(X)-r(X)\|_{\max}\leq\Bigl \| \mathrm{diag}\Bigl (U_X \Bigl ( I_\rho -  \Sigma_X^2 \Bigr ) U_X^T \Bigr) \Bigr \|\\
  &\quad =\Bigl \| \mathrm{diag}\Bigl (U_X  \Bigl (I_\rho - \Sigma_X^2 \Bigr ) U_X^T  \Bigr) \Bigr \|\\
  &\quad \leq |1-\lambda_{1}(\Sigma_X)^2
 |\vee|1-\lambda_{\rho}(\Sigma_X)^2|.
 \end{split}
\end{align}
On the other hand,
\begin{align}
\label{eqn:Yboundin1}
\begin{split}
  & \| l(Y)-r(Y)\|_{\max}= \Bigl \| \mathrm{diag} \Bigl ( U_Y \Bigl (  I_n -  \Sigma_Y^2 \Bigr ) U_Y^T \Bigr )\Bigr \|\\
  &\geq \frac{1}{kd}\Big|\text{Trace}\Big(U_Y \Bigl ( I_n -  \Sigma_Y^2 \Bigr ) U_Y^T\Big)\Big|= \frac{1}{kd}\Big|n-\sum_{i=1}^{n}\lambda_{i}(\Sigma_Y)^2 \Big|.
 \end{split}
\end{align}
Besides, lemma 4.4 in \cite{StefanVolkein} suggests that
\begin{align}\label{eqn:lambdainq}
   \lambda_{n-\rho+i}(\Sigma_Y^2) \leq \lambda_i(\Sigma_X^2)\leq \lambda_i(\Sigma_Y^2)
\end{align}
for $i\in [\rho]$. Then under the condition \eqref{eqn:rhocondin4}, the upper bound of \eqref{eqn:Xbound1} is 
\begin{align*}
    |1-\lambda_{1}(\Sigma_X)^2
 |\vee|1-\lambda_{\rho}(\Sigma_X)^2|= \lambda_{1}(\Sigma_X)^2-1\leq \lambda_{1}(\Sigma_Y)^2-1.
\end{align*}
Condition \eqref{eqn:rhocondin5} suggests that
$$\frac{1}{kd}\Big(\sum_{i=1}^{n}\lambda_{i}(\Sigma_Y)^2-n\Big)\geq \lambda_{1}(\Sigma_Y)^2-1,$$
which also implies that the lower bound of \eqref{eqn:Yboundin1} is $\frac{1}{kd}\Big(\sum_{i=1}^{n}\lambda_{i}(\Sigma_Y)^2-n\Big)$.
\end{proof}

\begin{proof}[Proof of Proposition \ref{prop:SimErr3-second}]
We have the normal equations
\begin{align*}
X^TX r=X^TZ^{-\frac{1}{2}}(D^T)^\dagger b=\Psi^T b,
\end{align*}
and 
\begin{align*}
X^TSS^TX\hat r = \Psi^T b.
\end{align*}
Subtracting the latter equation from the first one gives
\begin{align*}
&X^TSS^TX(r-\hat r)=-X^T(I-SS^T)X r.
\end{align*}
Taking $2$-norm yields
\begin{align*}
&\lambda_{\min}(\hat G)\|r- \hat r\| \leq \|G-\hat G\|\| r\|.
\end{align*}

Assume that matrix $\hat G$ is invertible. For the estimation of $\lambda_{\min}(\hat G)$, which is exactly $\lambda_\rho(\hat G )$, defining $\gamma\dot=\frac{\epsilon}{\epsilon+1}$ and following similar arguments as in the proof of Proposition \ref{prop:min} gives
\begin{align*}
\lambda_\rho(\hat G)\geq \lambda_\rho(G)+\lambda_\rho(\hat G-G)\geq (1-\gamma)\lambda_\rho(G)
\end{align*}
and
$$\|\hat G -G\|\leq \gamma\lambda_\rho(G)$$
with probability at least $1-\min\Big\{1,\frac{\mathbb{E}_{\xi}[\|G-\hat G\|_F]}{\gamma\lambda_\rho(G)}\Big\}$. Besides, based on the assumptions of $c$ and $\delta$, we have through Proposition \ref{prop:minsigma} that
$$\frac{\mathbb{E}_{\xi}[\|G-\hat G\|_F]}{\gamma\lambda_\rho(G)}\leq \frac{\sqrt{\mathbb{E}_{\xi}[\|G-\hat G\|^2_F]}}{\gamma\lambda_\rho(G)}\leq \frac{\sqrt{(\sum_{\ell=1}^{kd} z_\ell \|D_{(\ell)*}\|_2)^2- \|G\|^2_F  } }{\sqrt{c}\gamma\lambda_\rho(G)}\leq \delta.$$
Thus the probability above can be lower-bounded by $1-\delta$. In summary, these estimations lead to 
\begin{align*}
&\|\Psi r - \Psi \hat r \|\leq \|r - \hat r\|\leq \frac{\gamma}{1-\gamma}\|r\|\leq \epsilon\|\Psi r\|
\end{align*}
with probability $1-\delta$ for any $\epsilon,\delta\in (0,1)$ with $c$ chosen to be
$$(1+\frac 1 \epsilon)^2\frac{ \bigl((\sum_{\ell=1}^{kd} z_\ell \|(D\Psi)_{(\ell)*}\|)^2- \|G\|^2_F  \bigr)}{\delta\lambda_{\min}(G)^2}\leq c. $$
\end{proof}

\newcommand{\etalchar}[1]{$^{#1}$}

\end{document}